\documentclass[10pt, leqno]{amsart}
\usepackage{array}
\usepackage{longtable}
\usepackage{url}
\usepackage{enumerate}
\usepackage{amsmath}
\usepackage{amssymb}
\usepackage{amsthm}
\usepackage{amsfonts}
\usepackage{tikz,pgf}
\usepackage{graphicx}
  \usepackage{ifthen}
 \newcommand{\mymarginpar}[1]{%
    \marginpar{\ifthenelse{\isodd{\arabic{page}}}{\flushleft 
#1}{\flushright #1}}}

 \newcommand{\eps}{\varepsilon}           
 
 \newcommand{\IC}{\mathbb{C}}

 \newcommand{\IN}{\mathbb{N}}

 \newcommand{\IR}{\mathbb{R}}                  
                   
 \newcommand{\IT}{\mathbb{T}}                  
 \newcommand{\IZ}{\mathbb{Z}}

\newcommand{\CO}{\mathcal{O}}

\newcommand{\CT}{\mathcal{T}}

\newcommand{\CK}{\mathcal{K}}

\newcommand{\cstar}{$C^*$}

 \theoremstyle{plain} 
 \newtheorem{Theorem}{Theorem}[section]
 \newtheorem{Lemma}[Theorem]{Lemma}
 \newtheorem{Proposition}[Theorem]{Proposition}
 
 \newtheorem{Corollary}[Theorem]{Corollary}
 
 \theoremstyle{definition} 
 \newtheorem{Definition}[Theorem]{Definition}
 \newtheorem{Remark}[Theorem]{Remark}

\begin{document}

\author{Jack Spielberg}
\title{$C^*$-algebras for categories of paths associated to the Baumslag-Solitar groups}
\date{28 November 2011}
\address{School of Mathematical and Statistical Sciences \\ Arizona State University \\ P.O. Box 871804 \\ Tempe, AZ 85287-1804}
\email{jack.spielberg@asu.edu}

\begin{abstract}
In this paper we describe the $C^*$-algebras associated to the Baumslag-Solitar groups with the ordering defined by the usual presentations.  These are Morita equivalent to the crossed product $C^*$-algebras obtained by letting the group act on its \textit{directed boundary}.  We use the method of categories of paths to define the algebras, and to deduce the presentation by generators and relations.  We obtain a complete description of the Toeplitz algebras, and we compute the $K$-theory of the Cuntz-Kreiger algebras.
\end{abstract}

\maketitle

\section{Introduction}

In this paper we study the $C^*$-algebras associated to the Baumslag-Solitar groups with the ordering defined by the usual presentation.  We show that in most cases this is a quasi-lattice ordering in the sense of \cite{nica}.  However we use the notion of \textit{category of paths} of \cite{spi} to describe the $C^*$-algebras.  This is a construction of Toeplitz and Cuntz-Krieger algebras in a very general setting that includes ordered groups, higher-rank graphs, and many other examples of $C^*$-algebras obtained from oriented combinatorial objects.  The chief virtue of this construction is that a presentation of the algebras by generators and relations is obtained  naturally from the category, eliminating the guesswork typically needed for their identification.  For example, if a free semigroup acts on its $\ell^2$-space, it is easy to observe that the generators define isometries having pairwise orthogonal ranges spanning a codimension-one subspace.  The rank-one projection onto this subspace generates the compact operators when pushed around by the isometries, and in the quotient, the classes of isometries have range projections adding to the identity.  Cuntz's theorem (\cite{cun}) shows that this quotient algebra is uniquely defined by this presentation.  For other semigroups, however, it may not be obvious what is the correct quotient, even if the compact operators are present in the algebra generated by the isometries on $\ell^2$.  Moreover, it may also be unclear what relations ought to be used to define the quotient.  In the present case, we obtain generators and relations for a $C^*$-algebra Morita equivalent to the crossed product algebra associated to the action of a Baumslag-Solitar group on its \textit{directed boundary}.  These presentations turn out to coincide with certain examples obtained by Katsura in his work on topological graphs (\cite{kat}).  Our method gives a new approach to the description of these algebras by generators and relations, and also gives the ideal structure of the Toeplitz versions.  In the case of the solvable examples $BS(1,d)$, this is a well-known example with core isomorphic to a Bunce-Deddens algebra (\cite{alsulami}).  The presentations of the Baumslag-Solitar groups can be thought of as arising from their status as the fundamental group of a graph of groups:  the graph consists of one vertex and one edge,  with infinite cyclic groups attached to both vertex and edge.  The Bass-Serre theory requires the choice of an orientation for this graph; this amounts to a direction on the Cayley graph (\cite{serre}).  In this case, the directed Cayley graph is a category of paths, and we use the theory developed in \cite{spi}  to study the action of the group on the boundary (see \cite{spi}, Example 8.5).

The organization of this paper is as follows.  We first use the HNN structure of a Baumslag-Solitar group to prove that, as an ordered group (relative to the subsemigroup defined by the presentation), it is suitable for our constructions.  In the process we define an associated odometer-like action by the second generator.  We then identify the maximal directed hereditary subsets of the category, and hence the boundary of the semigroup.  We prove that the restriction of the groupoid to the boundary of the category is amenable.  We then identify all directed hereditary subsets, and prove that the entire groupoid is amenable.  From this we deduce the generators and relations for the Cuntz-Krieger algebra.  We then compute the $K$-theory of the algebra by first studying the fixed-point algebra by the gauge action, and then applying the Pimsner-Voiculescu exact sequence.  We mention that the degree functors for these categories are degenerate, and thus the fixed-point algebras are not AF.  Finally we use theorems from \cite{spi} to establish the fundamental structural properties of these algebras.

We conclude this introduction with a description of the definitions and results from \cite{spi} needed for the rest of the paper.  We identify the objects of a category with the identity morphisms in that category, and we use juxtaposition to indicate composition of morphisms.  Morphisms are referred to as \textit{paths}, and objects as \textit{vertices}.  We use $s$ and $r$ to denote the \textit{source} and \textit{range} of morphisms, and $\Lambda^0$ for the vertices in the category $\Lambda$.  A \textit{category of paths} is a small category satisfying

\begin{enumerate}
\item \label{d.categoryofpathsa}
$\alpha\beta = \alpha\gamma$ implies $\beta = \gamma$ \textit{(left-cancellation)}.
\item \label{d.categoryofpathsb} $\beta\alpha = \gamma\alpha$ implies $\beta = \gamma$ \textit{(right-cancellation)}.
\item \label{d.categoryofpathsc}
$\alpha\beta = s(\beta)$ implies $\alpha = \beta = s(\beta)$ \textit{(no inverses)}.
\end{enumerate}

For any $\alpha \in \Lambda$ we define the \textit{left shift} $\sigma^\alpha : \alpha\Lambda \to s(\alpha)\Lambda$ by $\sigma^\alpha(\alpha\beta) = \beta$  ($\sigma^\alpha$ is well-defined by left-cancellation).  The \textit{right shift map} $\beta \in s(\alpha)\Lambda \mapsto \alpha\beta \in r(\alpha)\Lambda$ is the inverse of $\sigma^\alpha$.  We say that $\beta$ \textit{extends} $\alpha$ if there exists $\alpha' \in \Lambda$ such that $\beta = \alpha \alpha'$.  (We may express this by writing $\beta \in \alpha\Lambda$.)  It follows from the definition that this is a partial order on $\Lambda$.  If $\beta$ is an extension of $\alpha$, we call $\alpha$ an \textit{initial segment} of $\beta$.  The set of initial segments of $\beta$ is denoted $[\beta]$.  We write $\alpha \Cap \beta$ ($\alpha$ \textit{meets} $\beta$) if $\alpha\Lambda \cap \beta\Lambda \not= \emptyset$, and $\alpha \perp \beta$ ($\alpha$ is \textit{disjoint from} $\beta$) otherwise.  We let $\alpha \vee \beta$ denote the set of \textit{minimal common extensions} of $\alpha$ and $\beta$, i.e. the minimal elements of $\alpha\Lambda \cap \beta\Lambda$.  For a subset $F \subseteq \Lambda$ we let $\bigvee F$ denote the set of minimal common extensions of the elements of $F$.  We say that $\Lambda$ is \textit{finitely aligned} if for every pair of elements $\alpha$, $\beta \in \Lambda$, there is a finite subset $G$ of $\Lambda$ such that $\alpha\Lambda \cap \beta\Lambda = \bigcup_{\eps \in G} \eps\Lambda$.  It follows that we may take $G = \alpha \vee \beta$.  In this paper we will only consider (certain) finitely aligned examples.

A subset $C \subseteq \Lambda$ is \textit{directed} if for all $\alpha$, $\beta \in C$ there is $\gamma \in C$ extending both $\alpha$ and $\beta$.  $C$ is \textit{hereditary} if $[\alpha] \subseteq C$ for every $\alpha \in C$.  The collection of all directed hereditary subsets of $\Lambda$ is denoted $\Lambda^*$; the set of maximal elements of $\Lambda^*$ is denoted $\Lambda^{**}$.  A directed hereditary set is \textit{finite} if it contains a maximal element; in this case it must be of the form $[\alpha]$ for some $\alpha \in \Lambda$.  Otherwise it is \textit{infinite}.  We define a topology on $\Lambda^*$ as follows.  For $\alpha \in \Lambda$, and $\beta_1$, $\ldots$, $\beta_n \in \alpha \Lambda \setminus \{\alpha\}$, let $E = \alpha \Lambda \setminus \bigcup_{i = 1}^n \beta_i \Lambda$.  Set 
\[
\widehat{E} = \{ C \in \Lambda^* : E \supseteq C \cap \gamma \Lambda \text{ for some } \gamma \in C \}.
\]
Then the $\widehat{E}$ form a base of compact-open sets for a locally compact Hausdorff topology.  The \textit{boundary} of $\Lambda$ is defined to be $\partial \Lambda = \overline{\Lambda^{**}}$.

We define a groupoid with unit space $\Lambda^*$ as follows.  First define a relation on $\Lambda \times \Lambda \times \Lambda^*$ by $(\alpha, \beta, x) \sim (\alpha', \beta', x')$ if there are $y \in \Lambda^*$ and $\gamma$, $\gamma' \in \Lambda$ such that $x = \gamma y$, $x' = \gamma' y$, $\alpha \gamma = \alpha' \gamma'$, and $\beta \gamma = \beta' \gamma'$.  Then $\sim$ is an equivalence relation.  The set of equivalence classes becomes a locally compact Hausdorff \'etale groupoid, where the set of composable pairs is
\[
G^2 = \bigl\{ \bigl([\alpha,\beta,x],\;
[\gamma,\delta,y]\bigr) : \beta x = \gamma y \bigr\},
\]
and inversion is given by $[\alpha,\beta,x]^{-1} = [\beta,\alpha,x]$.   Multiplication $G^2 \to G$ is given as follows.  Let $\bigl([\alpha,\beta,x],$\allowbreak$ \; [\gamma,\delta,y]\bigr) \in G^2$.  Since $\beta x = \gamma y$, it follows from Lemma 4.12 of \cite{spi} that there are $z$, $\xi$, and $\eta$ such that $x = \xi z$, $y = \eta z$, and $\beta\xi = \gamma\eta$.  Then
\[
[\alpha,\beta,x] \, [\gamma,\delta,y]\ =\ [\alpha\xi,\delta\eta,z].
\]
A base of compact-open sets for $G$ is given by the sets $[\alpha, \beta, \widehat{E}] = \{ [\alpha,\beta,x] : x \in \widehat{E} \}$.  

The $C^*$-algebras of $\Lambda$ are defined as $\CT C^*(\Lambda) = C^*(G)$ (the \textit{Toeplitz} $C^*$-algebra), and $C^*(\Lambda) = C^*(G|_{\partial \Lambda})$ (the \textit{Cuntz-Krieger} algebra).  We have the following theorem giving generators and relations for $\CT C^*(\Lambda)$ (\cite{spi}, Theorem 6.3). 

\begin{Theorem}
\label{t.toeplitzrelations}
Let $\Lambda$ be a finitely aligned category of paths.  The representations of $\CT C^*(\Lambda)$ are in one-to-one correspondence with the families $\{T_\alpha : \alpha \in \Lambda\}$ of Hilbert space operators satisfying the relations
\begin{enumerate}
\item $T_\alpha^* T_\alpha = T_{s(\alpha)}$.
\label{t.toeplitzrelations.a}
\item $T_\alpha T_\beta = T_{\alpha\beta}$, if $s(\alpha) = r(\beta)$.
\label{t.toeplitzrelations.b}
\item $T_\alpha T_\alpha^* T_\beta T_\beta^*
= \bigvee_{\gamma \in \alpha \vee \beta} T_\gamma T_\gamma^*$.
\label{t.toeplitzrelations.c}
\end{enumerate}
\end{Theorem}

In order to describe the analogous theorem for $C^*(\Lambda)$, we need the notion of \textit{exhaustive} set.  Let $v \in \Lambda^0$.  A subset $F \subseteq v\Lambda$ is exhaustive \textit{(at $v$)} if for every $\alpha \in v\Lambda$ there exists $\beta \in F$ such that $\alpha \Cap \beta$.  The exhaustive sets can be used to characterize the points of $\partial \Lambda$ (\cite{spi}, Theorem 7.8).  We have the following theorem giving generators and relations for $C^*(\Lambda)$ (\cite{spi}, Theorem 8.2).

\begin{Theorem}
\label{t.restrictionboundary}
Let $\Lambda$ be a countable finitely aligned category of paths.  Assume that $G$ is amenable.  The representations of $C^*(\Lambda)$ are in one-to-one correspondence with the families $\{S_\alpha : \alpha \in \Lambda\}$ of Hilbert space operators satisfying the relations
\begin{enumerate}
\item $S_\alpha^* S_\alpha = S_{s(\alpha)}$.
\label{t.restrictionboundary.a}
\item $S_\alpha S_\beta = S_{\alpha\beta}$, if $s(\alpha) = r(\beta)$.
\label{t.restrictionboundary.b}
\item $S_\alpha S_\alpha^* S_\beta S_\beta^*
= \bigvee_{\gamma \in \alpha \vee \beta} S_\gamma S_\gamma^*$.
\label{t.restrictionboundary.c}
\item $S_v = \bigvee_{\beta \in F} S_\beta S_\beta^*$ if $F $ is finite exhaustive at $v$.  (Equivalently, $0 = \prod_{\delta \in F} (S_v - S_\delta S_\delta^*)$.)
\label{t.restrictionboundary.d}
\end{enumerate}
\end{Theorem}

The last theorem requires amenability of the groupoid.  The usual way of obtaining this is to decompose $C^*(\Lambda)$ by means of a cocycle to a discrete abelian group (called a \textit{degree functor} in \cite{spi}, Definition 9.1).  We cite the following summary of the standard argument (\cite{spi}, Proposition 9.3).

\begin{Proposition}
\label{p.nuclearity}
Let $G$ be a locally compact Hausdorff \'etale groupoid, $Q$ a countable abelian group, and $c : G \to Q$ a continuous homomorphism.  Let $G^c = c^{-1}(0)$, also a locally compact Hausdorff \'etale groupoid.  Suppose that $G^c$ is amenable.  Then $G$ is amenable.
\end{Proposition}


An ordered group $(\Gamma,\Lambda)$ is called \textit{quasi-lattice ordered} if for each $t \in \Lambda \Lambda^{-1}$ there is an element $\alpha \in t \Lambda \cap \Lambda$ such that $t \Lambda \cap \Lambda = \alpha \Lambda$ (\cite{nica}).  This idea was generalized in \cite{spi}, Definition 8.6, as follows.  The ordered group $(\Gamma,\Lambda)$ is called \textit{finitely aligned} if for each $t \in \Gamma$ there is a finite set $F \subseteq t \Lambda \cap \Lambda$ such that $t \Lambda \cap \Lambda = \bigcup_{\alpha \in F} \alpha \Lambda$.  It is possible that $\Lambda$ is a finitely aligned category of paths even if $(\Gamma,\Lambda)$ is not a finitely aligned ordered group.  A weaker notion is given in \cite{spi}, Definition 8.11:  $(\Gamma,\Lambda)$ is \textit{locally finitely exhaustible} if for each $t \in \Gamma$, there is a finite set $F \subseteq t \Lambda \cap \Lambda$ such that every element of $t \Lambda \cap \Lambda$ meets some element of $F$.  In Lemma 8.10 of \cite{spi} it is shown how to define a locally compact $\Gamma$-space, $\partial(\Gamma,\Lambda)$, which is the \textit{directed boundary} of the ordered group.  The following appears in \cite{spi} as part of Theorem 8.13 and Corollary 8.17.

\begin{Theorem}
Let $(\Gamma,\Lambda)$ be a countable ordered group.  Suppose that $\Lambda$ is finitely aligned as a category of paths.  Then $\partial (\Gamma,\Lambda)$ is locally compact Hausdorff if and only if $(\Gamma,\Lambda)$ is locally finitely exhaustible.  Moreover, $C^*(\Lambda)$ is Morita equivalent to the crossed product algebra $C_0(\partial(\Gamma,\Lambda)) \times \Gamma$.
\end{Theorem}


\section{The category of paths of a Baumslag-Solitar group}

For nonzero integers $c$ and $d$ we consider the group $\Gamma = \langle a,b \bigm| a b^c = b^d a \rangle$ (\cite{baumsol}).  Since $a b^{-c} = b^{-d} a$ is an equivalent relation, we may as well assume that $c$ and $d$ are not both negative.  We will consider separately the cases $cd > 0$ and $cd < 0$.  Thus we formulate the situation for a pair of positive integers $c$ and $d$, and let the group be defined by the relation $a b^c = b^d a$, or by the relation $a b^c = b^{-d} a$.  The case $cd > 0$ will be further divided accordingly as $c \ge d$ or $c < d$, giving three cases overall:

\begin{description}
\item[Case (BS1)] $\Gamma = \langle a,b \bigm| a b^c = b^d a \rangle$ where $c \ge d \ge 1$.
\item[Case (BS2)] $\Gamma = \langle a,b \bigm| a b^c = b^d a \rangle$ where $d > c \ge 1$.
\item[Case (BS3)] $\Gamma = \langle a,b \bigm| a b^c = b^{-d} a \rangle$ where $c$, $d \ge 1$.
\end{description}

We denote by $\theta : G \to \IZ$ the  homomorphism given by $\theta(a) = 1$ and $\theta(b) = 0$.  We sometimes refer to $\theta(g)$ as the \textit{height} of $g$.  Let $\Lambda$ be the submonoid generated by $a$ and $b$, and let $B = \{ b^i : i \ge 0\}$ denote the submonoid generated by $b$.  The following may be found in \cite{scottwall}.


\begin{Proposition}
\label{p.scottwall}
Each element of $\Gamma$ has a unique represention in the form $b^{i_1} a^{\eps_1} \cdots b^{i_n} a^{\eps_n} b^q$, where $\eps_\mu \in \{ \pm 1\}$, $i_\mu \in [0,d)$ if $\eps_\mu = +1$, and $i_\mu \in [0,c)$ if $\eps_\mu = -1$.
\end{Proposition}

The standard from of the proposition is obtained by moving $b$'s to the right via $b^{kd} a = a b^{kc}$ and $b^{kc} a^{-1} = a^{-1} b^{kd}$.

\begin{Corollary}
\label{c.scottwall}
Let $t \in \Gamma$ have the form in Proposition \ref{p.scottwall}.  Then $t \in \Lambda$ if and only if $\eps_\mu = +1$ for all $\mu$, and $q \ge 0$ in cases (BS1) and (BS2), and in case (BS3) if $n = 0$.
\end{Corollary}

We note the following proposition for later use.  It follows from the same kind of arguments as gives Proposition \ref{p.scottwall} and its corollary.


\begin{Proposition}
\label{p.baumslagsolitar}
Each element $\alpha \in \Lambda$ has unique representations in the two forms
\begin{itemize}
\item[(L)] $\alpha = b^{i_0} a b^{i_1} a \cdots b^{i_{k - 1}} a b^p$, $i_\mu \in [0,d)$, $p \in \IZ$;
\item[(R)] $\alpha = b^q a b^{j_1} a b^{j_2} \cdots a b^{j_k}$, $j_\mu \in [0,c)$, $q \in \IZ$,
\end{itemize}
where $k = \theta(\alpha)$.
\end{Proposition}

\begin{Remark}
In cases (BS1) and (BS2), in forms (L) and (R) we have $p$, $q \ge 0$.  In case (BS3), if $\theta(\alpha) = 0$ we have $p = q \ge 0$.  It also follows that $b$ has infinite order in $G$.
\end{Remark}

The following corollary follows easily.


\begin{Corollary}
\label{c.initseq}
If $0 < k < \theta(\alpha)$, or if $k = \theta(\alpha) > 0$ and we are in case (BS1) or (BS2), then $\alpha$ has a unique initial segment of height $k$ in the form (L) of Proposition \ref{p.baumslagsolitar} with $p = 0$.  If $k = \theta(\alpha) > 0$ and we are in case (BS3), then there is a unique element $\alpha_0$ in the form (L) of Proposition \ref{p.baumslagsolitar} with $p = 0$, and $p_0 \in \IZ$, such that every initial segment of height $k$ has form (L) equal to $\alpha_0 b^q$ with $q \le p_0$.
\end{Corollary}


\begin{Lemma}
\label{l.bincaseethree}
In case (BS3), if $\alpha \in \Lambda$ and $\theta(\alpha) > 0$, then $\alpha b^p \in \Lambda$ for all $p \in \IZ$.  Moreover, $B \subseteq [\alpha]$.
\end{Lemma}

\begin{proof}
We can write $\alpha = \alpha_0 a b^m$ for $\alpha_0 \in \Lambda$ and $m \ge 0$.  Choose $n$ so that $nc + m + p \ge 0$.  Then using the relation $a = b^d a b^c$, we find that $\alpha b^p = \alpha_0 a b^{m + p} = \alpha_0 b^{nd} a b^{nc + m + p} \in \Lambda$.  For the second statement, we write $\alpha = b^i a \alpha_1$ for some $i \in \IN$ and $\alpha_1 \in \Lambda$.  Then $\alpha = b^{i + md} a b^{mc} \alpha_1$ for all $m \in \IN$.
\end{proof}

The following begins our study.

\begin{Lemma}
$\Lambda$ is a category of paths (\cite{spi}, Example 8.5).
\end{Lemma}

\begin{proof}
Note that $\theta(\alpha) \ge 0$ for all $\alpha \in \Lambda$. We note that if $\alpha$, $\beta \in \Lambda$ are such that $\alpha \beta = e$, then $\theta(\alpha) + \theta(\beta) = 0$, and hence $\theta(\alpha) = \theta(\beta) = 0$.  But then $\alpha$, $\beta \in B$.  But then we must have $\alpha = \beta = e$.   Hence $\Lambda$ is a category of paths.  
\end{proof}

We will prove that $\Lambda$ is finitely aligned, and that in cases (BS1) and (BS2), the ordered group $(\Gamma, \Lambda)$ is quasi-lattice ordered in the sense of \cite{nica} (Theorem \ref{t.qlo} below).  (The case (BS3) is slightly different --- see Lemma \ref{l.locfinexh}.)  The argument varies by case.  We require lemmas for cases (BS1) and (BS2).

\begin{Lemma}
\label{l.clessthand}
Suppose that we are in case (BS2).  Fix $k \ge 1$.  For $i = (i_0, \ldots, i_k) \in [0,c)^{k + 1}$, let $\alpha(i) = b^{i_0} a b^{i_1} a \cdots b^{i_k} a$.  There are maps $\psi : [0,c)^{k + 1} \to [0,c)^{k + 1}$ and $r : [0,c)^{k + 1} \to \IZ^+$ such that
\begin{enumerate}
\item $b^{d r(i)} \alpha(\psi(i)) = \alpha(i) b^c$ and $\psi(i)_0 = i_0$.
\label{l.clessthand.a}
\item If $q < c$ then $b^{i_0 + 1} \not\in [\alpha(i) b^q]$.
\label{l.clessthand.b}
\item If $i_0 < m \le i_0 + d r(i)$, then $\alpha(i) b^c = b^{d r(i)} \alpha(\psi(i))$ is the unique minimal common extension of $b^m$ and $\alpha(i)$. 
\label{l.clessthand.c}
\item For $h \ge 0$, if 
\[
i_0 + d \sum_{\mu=0}^{h-1} r(\psi^\mu(i)) < m \le i_0 + d \sum_{\mu=0}^h r(\psi^\mu(i))
\]
then 
\[
\alpha(i) b^{(h+1)c} = b^{d \sum_{\mu=0}^h r(\psi^\mu(i))} \alpha(\psi^{h+1}(i))
\]
is the unique minimal common extension of $b^m$ and $\alpha(i)$.
\label{l.clessthand.d}
\end{enumerate}
\end{Lemma}

\begin{proof}
For \eqref{l.clessthand.a}, the maps $\psi$ and $r$ are defined by the unique form (R) of $\alpha(i) b^c$.  That $r(i) \ge 1$ follows from the assumption $c < d$.  It follows from the defining relation of $G$ that if $i_0 = 0$ then $\alpha(i)_0 = 0$.  Hence by the uniqueness of form (L), $\psi(i)_0 = i_0$ for all $i$.  For \eqref{l.clessthand.b}, note that if $q < c$ then $\alpha(i) b^q$ is already in form (R). We next prove \eqref{l.clessthand.c}.  Let $m$ be as in the statement.  By \eqref{l.clessthand.a} we know that $\alpha(i) b^c \in b^m \Lambda \cap \alpha(i) \Lambda$.  Let $\beta \in b^m \Lambda \cap \alpha(i) \Lambda$.  Then $\beta = \alpha(i) \gamma$ for some $\gamma$.  Write $\gamma = b^q a b^{\ell_1} \cdots a b^{\ell_n}$ with $\ell_\mu \in [0,c)$.  Then the form (R) of $\beta$ equals (form (R) of $\alpha(i) b^q) \cdot a b^{\ell_1} \cdots a b^{\ell_n}$.  Therefore $b^m \in [\beta]$ if and only if $b^m \in [\alpha(i) b^q]$.  This occurs if and only if $q \ge c$, by \eqref{l.clessthand.b}.  Therefore $\beta \in \alpha(i) b^c \Lambda$.  Finally, we prove \eqref{l.clessthand.d}, by induction on $h$.  The case $h = 0$ is \eqref{l.clessthand.c}.  Suppose it is true for $h$, and let
\[
i_0 + d \sum_{\mu=0}^h r(\psi^\mu(i)) < m \le i_0 + d \sum_{\mu=0}^{h+1} r(\psi^\mu(i)).
\]
First we note that
\[
\alpha(i) b^{(h+2)c}
= b^{d \sum_{\mu=0}^h r(\psi^\mu(i))} \alpha(\psi^{h+1}(i)) b^c
= b^{d \sum_{\mu=0}^{h+1} r(\psi^\mu(i))} \alpha(\psi^{h+2}(i))
\in b^m \Lambda.
\]
Thus $\alpha(i) b^{(h+2)c}$ is a common extension of $b^m$ and $\alpha(i)$.

Now let $\beta \in b^m \Lambda \cap \alpha(i) \Lambda$ be any common extension.  Then
\[
\beta \in b^{m - d r(\psi^h(i))} \Lambda \cap \alpha(i) \Lambda.
\]
Since $i_0 + d \sum_{\mu=0}^{h-1} r(\psi^\mu(i)) < m - d r(\psi^h(i))$, the inductive hypothesis implies that $\beta \in \alpha(i) b^{(h+1)c} \Lambda$.  Define $\beta'$ by
\[
\beta = b^{d \sum_{\mu=0}^h r(\psi^\mu(i))} \beta'.
\]
Then $\beta' \in b^{i_0 + 1} \Lambda \cap \alpha(\psi^{(h+1)}(i)) \Lambda$.  By \eqref{l.clessthand.c}, $\beta' \in \alpha(\psi^{h+1}(i)) b^c \Lambda$.  Therefore $\beta \in \alpha(i) b^{(h+2)c} \Lambda$.
\end{proof}

We next consider case (BS1).  We will use the following notation:  for $i \in [0,d)^\IN$ and $k \in \IN$, let $\alpha_k(i) = b^{i_0} a \cdots b^{i_k} a$.

\begin{Lemma}
\label{l.cged}
Suppose that we are in case (BS1).  There are maps $\phi : [0,d)^\IN \to [0,d)^\IN$ and $r : [0,d)^\IN \to (\IZ^+)^\IN$ such that for all $k \in \IN$,
\begin{enumerate}
\item $b^d \alpha_k(\phi(i)) = \alpha_k(i) b^{c r(i)_k}$.  $\phi$ and $r$ are uniquely determined by this condition, and $\phi(i)_k$, $r(i)_k$ depend only on $i_0$, $\ldots$, $i_k$.
\label{l.cged.a}
\item If $q < c r(i)_k$ then $b^{i_0 + 1} \not\in [\alpha_k(i) b^q]$.
\label{l.cged.b}
\item If $i_0 < m \le i_0 + d$ then $\alpha_k(i) b^{c r(i)_k}$ is the unique minimal common extension of $\alpha_k(i)$ and $b^m$.
\label{l.cged.c}
\item If $i_0 + (h-1)d < m \le i_0 + hd$ then $\alpha_k(i) b^{c \sum_{\mu=0}^{h-1} r(\phi^\mu(i))_k} = b^{hd} \alpha_k(\phi^h(i))$ is the unique minimal common extension of $\alpha_k(i)$ and $b^m$.
\label{l.cged.d}
\item The map $\phi$ is a homeomorphism (for the product topology on $[0,d)^\IN$).
\label{l.cged.e}
\end{enumerate}
\end{Lemma}

\begin{proof}
\eqref{l.cged.a}: We define $\phi$ and $r$ inductively.  Put $\phi(i)_0 = i_0$ and $r(i)_0 = 1$.  Then 
\[
b^d \alpha_0(\phi(i)) = b^d b^{i_0} a = b^{i_0} a b^c = \alpha_0(i) b^{c r(i)_0}.
\]
It is clear that $\phi(i)_0$ and $r(i)_0$ are determined uniquely.  Suppose that $\phi(i)_\mu$ and $r(i)_\mu$ have been defined for $\mu < k$ so that $b^d \alpha_\mu(\phi(i)) = \alpha_k(i) b^{c r(i)_\mu}$.  Define $\phi(i)_k \in [0,d)$ by
\[
\phi(i)_k \equiv i_k - c r(i)_{k-1} \pmod{d}.
\]
Then there is $r(i)_k \ge 0$ such that
\begin{equation} \label{e.phiandr}
\phi(i)_k - d r(i)_k = i_k - c r(i)_{k-1}. \tag{$*$}
\end{equation}
We note that 
\[
d r(i)_k = \phi(i)_k - i_k + c r(i)_{k - 1} \ge \phi(i)_k - i_k + c > \phi(i)_k,
\]
and hence that $r(i)_k \ge 1$.
Now,
\[
b^d \alpha_k(\phi(i))
= b^d \alpha_{k-1}(\phi(i)) b^{\phi(i)_k} a 
= \alpha_{k-1}(i) b^{c r(i)_{k-1}} b^{\phi(i)_k} a 
= \alpha_{k-1}(i) b^{i_k + d r(i)_k} a 
= \alpha_k(i) b^{c r(i)_k}.
\]
It is clear that $\phi(i)_k$ and $r(i)_k$ are uniquely determined by the equation $b^{\phi(i)_k + c r(i)_{k-1}} a = b^{i_k} a b^{c r(i)_k}$.  Moreover, the construction shows that they depend only on $i_0$, $\ldots$, $i_k$.

\eqref{l.cged.b}:  We again prove this by induction on $k$.  Let $q < c r(i)_0 = c$.  Then $\alpha_0(i) b^q = b^{i_0} a b^q$ is in form (R), and hence does not extend $b^{i_0 + 1}$.  Suppose that the result is true for $k-1$.  If $q < c r(i)_k$, write $q = sc + t$, with $t \in [0,c)$.  Then $s < r(i)_k$.  We have
\[
\alpha_k(i) b^q 
= \alpha_{k-1}(i) b^{i_k} a b^{sc} b^t
= \alpha_{k-1}(i) b^{i_k} b^{sd} a b^t.
\]
By \eqref{e.phiandr}, 
$i_k + sd 
\le i_k + (r(i)_k - 1)d
= i_k + r(i)_k d - d
= \phi(i)_k + c r(i)_{k-1} - d
< c r(i)_{k-1}$.
Then the inductive hypothesis implies that $\alpha_{k-1}(i) b^{i_k + sd} \not\in b^{i_0 + 1} \Lambda$.  By uniqueness of form (R), it follows that $\alpha_k(i) b^q \not\in b^{i_0 + 1} \Lambda$.  

\eqref{l.cged.c}:  Let $\beta \in \alpha_k(i) \Lambda$.  Write $\sigma^{\alpha_k(i)} \beta = b^q a b^{e_1} \cdots a b^{e_\ell}$ in form (R), so $e_\mu \in [0,c)$.  By uniqueness of form (R), $\beta \in b^m \Lambda$ if and only if $\alpha_k(i) b^q \in b^m \Lambda$.  By \eqref{l.cged.a} and \eqref{l.cged.b} this occurs if and only if $q \ge c r(i)_k$, that is, if and only if $\beta \in \alpha b^{c r(i)_k} \Lambda$.

\eqref{l.cged.d}:  The case $h = 1$ is given by \eqref{l.cged.c}.  Suppose the result is true for $h - 1$.  Let $i_0 + (h - 1)d < m \le i_0 + hd$.  The inductive hypothesis implies that $b^{(h-1)d} \alpha_k(\phi^{h - 1}(i))$ is the minimal common extension of $b^{m - d}$ and $\alpha_k(i)$.  Let $\beta \in b^m \Lambda \cap \alpha_k(i) \Lambda$.  Then $\beta \in b^{(h - 1)d} \Lambda \cap \alpha_k(i) \Lambda$.  Therefore $\beta \in b^{(h - 1)d} \alpha_k(\phi^{h - 1}(i))$.  Writing $\beta = b^{(h - 1)d} \beta'$, we have that $\beta' \in b^{m - (h - 1)d} \Lambda \cap \alpha_k(\phi^{h - 1}(i)) \Lambda$.  Since $i_0 < m - (h - 1)d \le i_0 + d$, \eqref{l.cged.c} implies that $\beta' \in \alpha_k(\phi^{h - 1}(i)) b^{c r(\phi^{h - 1}(i))_k} = b^d \alpha_k(\phi^h(i))$.  Therefore $\beta \in b^{hd} \alpha_k(\phi^h(i))$.  Since $b^{hd} \alpha_k(\phi^h(i))$ extends $b^m$ and $\alpha_k(i)$, the result is proved.

\eqref{l.cged.e}:  Given $j \in [0,d)^\IN$, the unique form (L) gives $i \in [0,d)^\IN$ and $r \in [0,d)^{\IZ^+}$ such that $b^d \alpha_k(j) = \alpha_k(i) b^{c r_k}$.  It follows that $j = \phi(i)$, and so $\phi$ is onto.  By the uniqueness of forms (L) and (R), $j$ and $i$ uniquely determine each other, so that $\phi$ is bijective.  The continuity of $\phi$ follows from the last part of \eqref{l.cged.a}.  Since $[0,d)^\IN$ is compact, $\phi$ is a homeomorphism.
\end{proof}

\begin{Proposition}
\label{p.qlo}
Let $\alpha = b^{e_0}a \cdots b^{e_s} a b^{e_{s+1}}$ and $\beta = b^{f_0}a \cdots b^{f_t}a b^{f_{t+1}}$ in form (L), i.e. with $e_\mu$, $f_\nu \in [0,d)$ for $\mu \le s$ and $\nu \le t$.  Then $\alpha \Cap \beta$ if and only if $e_\mu = f_\mu$ for $\mu \le \min\{s,t\}$.  In this case, $\alpha$ and $\beta$ have a unique minimal common extension.  (In particular, $\Lambda$ is finitely aligned.)
\end{Proposition}

\begin{proof}
The ``only if'' statement follows from Corollary \ref{c.initseq}.  For the converse, suppose that $s \le t$ and $e_\mu  = f_\mu$ for $\mu \le s$.   We first consider case (BS2).  Let $b^{f_{s+1}} a \cdots b^{f_{t+1}} = b^v a b^{i_1} \cdots b^{i_k} a b^q$ in form (R), i.e. with $q, i_\mu \in [0,c)$ for $1 \le \mu \le k$.  If $e_{s+1} \le v$, then $\beta \in \alpha \Lambda$, and the result holds.  Suppose instead that $v < e_{s+1}$.  Applying $\sigma^{b^{e_0}a \cdots b^{e_s} a b^v}$ to $\alpha$ and $\beta$, and letting $m = e_{s+1} - v$, we find that it is enough to show that the paths $b^m$ and $a b^{i_1} a \cdots b^{i_k}a b^q$ meet, and have a unique minimal common extension.  Choose $h$ as in Lemma \ref{l.clessthand}\eqref{l.clessthand.d}.  Then $a b^{i_1} a \cdots b^{i_k}a b^{(h+1)c}$ is the unique minimal common extension of $b^m$ and $a b^{i_1} a \cdots b^{i_k}a$.  Since $q < c$, this is also the unique minimal common extension of $b^m$ and $a b^{i_1} a \cdots b^{i_k}a b^q$.

Now we consider case (BS1).  If $e_{s+1} \le f_{s+1}$ then $\beta \in \alpha \Lambda$, and the result holds.  Suppose instead that $e_{s+1} > f_{s+1}$.  Applying $\sigma^{b^{e_0} a \cdots b^{e_s} a b^{f_{s+1}}}$ to $\alpha$ and $\beta$, and letting $m = e_{s+1} - f_{s+1}$, $k = t - s - 1$, $i_\mu = f_{\mu + s + 1}$ for $1 \le \mu \le k$, and $q = f_{t+1}$, it is enough to show that $b^m$ and $a b^{i_1} \cdots b^{i_k} a b^q$ meet and have a unique minimal common extension.  Choose $h$ such that $(h-1)d < m \le hd$.  By Lemma \ref{l.cged} there are $j_1$, $\ldots$, $j_k \in [0,d)$, and $r \ge 0$, such that $b^{hd} a b^{j_1} \cdots b^{j_k} a = a b^{i_1} \cdots b^{i_k} a b^{cr}$ is the unique minimal common extension of $b^m$ and $a b^{i_1} \cdots b^{i_k} a$.  If $q \le cr$ this is also the unique minimal common extension of $b^m$ and $a b^{i_1} \cdots b^{i_k} a b^q$.  If $q > cr$, then $a b^{i_1} \cdots b^{i_k} a b^q$ is the unique minimal common extension.

Finally we consider case (BS3).  If $s < t$, choose $m \ge 0$ such that $f_{s+1} - e_{s+1} + md \ge 0$.  Then
\[
\beta = b^{e_0} a \cdots b^{e_s} a b^{f_{s+1}} a \cdots
= \alpha b^{f_{s+1} - e_{s+1}} a \cdots
= \alpha b^{f_{s+1} - e_{s+1} + md} a b^{mc} \cdots
\in \alpha \Lambda.
\]
If $s = t$, then $\beta \in \alpha \Lambda$ if $e_{s+1} \le f_{s+1}$, and $\alpha \in \beta \Lambda$ if $e_{s+1} > f_{s+1}$.  (Thus in case (BS3), $\alpha \Cap \beta$ if and only if $\alpha$ and $\beta$ are comparable.)
\end{proof}


\begin{Theorem}
\label{t.qlo} 
In cases (BS1) and (BS2), the ordered group $(\Gamma,\Lambda)$ is quasi-lattice ordered.
\end{Theorem}

\begin{proof}
Let $t \in \Lambda \Lambda^{-1}$.  Let $t = \alpha \beta^{-1}$ with $\alpha$, $\beta \in \Lambda$.  Write $\alpha = \alpha_k(i) b^p$ and $\beta = \alpha_\ell(j) b^q$ in form (L).  Then $t = \alpha_k(i) b^{p - q} \alpha_\ell(j)^{-1}$.  If $p - q \equiv 0 \pmod c$, then $a b^{p - q} a^{-1} = b^n$, with $n \in \IZ$.  Then $t = \alpha_{k-1}(i) b^n \alpha_{\ell - 1}(j)^{-1}$. If $n \equiv 0 \pmod c$, we may repeat this procedure.  Continue this until no more such cancellation is possible.  Thus we may assume that $t = \alpha_k(i) b^n \alpha_\ell(j)^{-1}$ with $n \not\equiv 0 \pmod c$.

We first consider the case that $n \le 0$.  Let $\mu \in t \Lambda \cap \Lambda$.  There is $\nu \in \Lambda$ such that $t \nu = \mu$; that is,  $\alpha_k(i) b^n \alpha_\ell(j)^{-1} \nu = \mu$.  Since the reduced form of the left hand side is obtained by moving $b$'s to the right (see the remarks after Proposition \ref{p.scottwall}), and since no cancellation is possible across $b^n$, it follows that $b^n \alpha_\ell(j)^{-1} \nu \in \Lambda$.  Then by Corollary \ref{c.initseq} we have $\mu \in \alpha_k(i) \Lambda$.  Since $\alpha_k(i) \in t \Lambda \cap \Lambda$, we have $t \Lambda \cap \Lambda = \alpha_k(i) \Lambda$.

Finally, if $n > 0$ we have $t^{-1} = \alpha_\ell(j) b^{-n} \alpha_k(i)^{-1}$.  By the previous argument we know that $t^{-1} \Lambda \cap \Lambda = \alpha_\ell(j) \Lambda$.  Let $\mu \in t \Lambda \cap \Lambda$, and let $\nu \in \Lambda$ with $t \nu = \mu$.  Then $t^{-1} \mu = \nu$, so $\nu = \alpha_\ell(j) \gamma$ for some $\gamma \in \Lambda$.  Then $\mu = t \alpha_\ell(j) \gamma = \alpha_k(i) b^n \gamma \in \alpha_k(i) b^n \Lambda$.  Since $\alpha_k(i) b^n \in t \Lambda \cap \Lambda$, we have $t \Lambda \cap \Lambda = \alpha_k(i) b^n \Lambda$.
\end{proof}

In case (BS3), it turns out that the ordered group $(\Gamma,\Lambda)$ is not usually finitely aligned, even though $\Lambda$ is finitely aligned (with unique minimal common extensions).  However $(\Gamma, \Lambda)$ is always locally finitely exhaustible.

\begin{Lemma}
\label{l.locfinexh}
Suppose we are in case (BS3).  Then $(\Gamma, \Lambda)$ is locally finitely exhaustible.  It is finitely aligned if and only if $c = 1$ (and in this case is quasi-lattice ordered).
\end{Lemma}

\begin{proof}
Let $t \in \Lambda \Lambda^{-1}$.  As in the proof of Theorem \ref{t.qlo}, we may assume that $t = \alpha_k(i) b^n \alpha_\ell(j)^{-1}$, where $n \not\equiv 0 \pmod c$.  We first assume that $\ell \ge 0$, (so that $\theta(\alpha_\ell(j)) > 0$).  Then $\alpha_k(i) = t \alpha_\ell(j) b^{-n} \in t \Lambda \cap \Lambda$.  If $\mu \in t \Lambda \cap \Lambda$, and $t \nu = \mu$ with $\nu \in \Lambda$, then we have $\alpha_k(i) b^n \alpha_\ell(j)^{-1} \nu = \mu$.  By Corollary \ref{c.initseq} there is $p_0 \in \IZ$ such that $\alpha_k(i) b^{q_0} \in [\mu]$.  Therefore $\alpha_k(i) \Cap \mu$.  Now assume that $\ell = -1$ (that is, that $\alpha_\ell(j) \in B$).  If $k \ge 0$, then $t \in \alpha_k(i) B \subseteq \Lambda$.  If $k = -1$, then $t \in B \cup B^{-1}$.  Then $t \Lambda \cap \Lambda = t \Lambda$, if $t \in B$, and $= \Lambda$ if $t \in B^{-1}$.

To see that $(\Gamma,\Lambda)$ is not finitely aligned if $c > 1$, consider $t = a b a^{-1}$.  Then $t \Lambda \cap \Lambda = \bigcup_{n \in \IZ} a b^n \Lambda$, but for all $n \in \IZ$, $a b^n \not\in \bigcup_{k > n} a b^k \Lambda$.

In the case that $c = 1$, it is easy to see that $\Lambda \Lambda^{-1} = \Lambda \cup \Lambda^{-1}$, so that $(\Gamma, \Lambda)$ is totally ordered, hence quasi-lattice ordered.
\end{proof}

We remark that the definition of $C^*(\Lambda)$, and hence the rest of the paper, requires only that $\Lambda$ be a finitely aligned category of paths (Proposition \ref{p.qlo}).  The previous results show that the interpretation as an algebra Morita equivalent to the crossed product of $\Gamma$ acting on its directed boundary is valid.


Lemma \ref{l.clessthand} served the purpose of proving that $(\Gamma, \Lambda)$ is quasi-lattice ordered in case (BS2).  We next give results analogous to Lemma \ref{l.cged} for cases (BS2) and (BS3).

\begin{Lemma}
\label{l.cltd} 
Assume that we are in case (BS2).
\begin{enumerate}
\item \label{l.cltd.a}
There are maps $\phi : [0,d)^\IN \to [0,d)^\IN$ and $r : [0,d)^\IN \to \{0,1\}^\IN$ such that for all $k \ge 0$ we have $b^d \alpha_k(\phi(i)) = \alpha_k(i) b^{c r(i)_k}$.
\item \label{l.cltd.b}
Let $\ell = \inf\{\mu \ge 1 : i_\mu \ge c\}$ (and $\ell = \infty$ if $i_\mu < c$ for $\mu \ge 1$).  Then $r(i)_k = 1$ if and only if $k < \ell$.
\item \label{l.cltd.c}
$\phi$ is a homeomorphism of $[0,d)^\IN$ for the product topology.
\end{enumerate}
\end{Lemma}

\begin{proof}
\eqref{l.cltd.a} and \eqref{l.cltd.b}: We define $\phi(i)$ by
\[
\phi(i)_\mu =
\begin{cases}
i_0, & \text{ if } \mu = 0 \\
i_\mu + d - c, & \text{ if } 0 < \mu < \ell \\
i_\ell - c, & \text{ if } \mu = \ell \\
i_\mu, & \text{ if } \mu > \ell.
\end{cases}
\]
Then $b^d \alpha_0(\phi(i)) = b^d b^{i_0} a = b^{i_0} a b^c = \alpha_0(i) b^c$.  Inductively, for $0 < \mu < \ell$, we have
\[
b^d \alpha_\mu(\phi(i))
= \alpha_{\mu - 1}(i) b^c b^{d - c + i_\mu} a
= \alpha_{\mu - 1}(i) b^{i_\mu} a b^c
= \alpha_\mu(i) b^c.
\]
If $\ell < \infty$,
\[
b^d \alpha_\ell(\phi(i))
= \alpha_{\ell - 1}(i) b^c b^{i_\ell - c} a
= \alpha_\ell(i),
\]
and for $\mu > \ell$,
\[
b^d \alpha_\mu(\phi(i))
= b^d \alpha_\ell(\phi(i)) b^{i_{\ell + 1}} a \cdots b^{i_\mu} a
= \alpha_\mu(i).
\]
\eqref{l.cltd.c}: Let $j \in [0,d)^\IN$.  Define $i \in [0,d)^\IN$ and $r \in \IN^\IN$ by form (L) of $b^d \alpha_k(j)$ for all $k$:
\[
b^d \alpha_k(j) = \alpha_k(i) b^{c r_k}.
\]
That $i_\mu$ is well-defined independently of the choice of $k \ge \mu$ follows from Corollary \ref{c.initseq}.  We first show that $r_k \in \{0,1\}$ for all $k$.  When $k = 0$ we have $b^d b^{j_0} a = b^{j_0} a b^c$, so $r_0 = 1$.  Suppose $r_{k-1} \in \{0,1\}$.  Then
\[
b^d \alpha_k(j)
= b^d \alpha_{k-1}(j) b^{j_k} a
=\alpha_{k-1}(i) b^{c r_{k-1} + j_k} a.
\]
If $r_{k-1} = 0$, then $b^d \alpha_k(j) = \alpha_{k-1}(i) b^{j_k} a$, so that $i_k = j_k$ and $r_k = 0$.  If $r_{k-1} = 1$, then $b^d \alpha_k(j) = \alpha_{k-1}(i) b^{c + j_k} a$.  Then, if $c + j_k \ge d$ we have $i_k = j_k + c - d$ and $r_k = 1$, while if $c + j_k < d$ we have $i_k = j_k + c$ and $r_k = 0$.

Notice that in the course of the last argument, we showed also that if $r_{k-1} = 0$ then $r_k = 0$.  Let $r_\mu = 1$ for $\mu \le k$.  Then
\[
\alpha_{k-1}(i) b^{i_k} a b^c
= \alpha_k(i) b^c
= b^d \alpha_k(j)
= b^d \alpha_{k-1}(j) b^{j_k} a
= \alpha_{k-1}(i) b^{c + j_k} a.
\]
Therefore $b^{i_k + d} a = b^{i_k} a b^c = b^{j_k + c} a$, and hence $i_k + d = j_k + c$.  Thus $i_k = c + j_k - d < c$.  It follows that $j = \phi(i)$, and thus $\phi$ is onto.  To see that $\phi$ is one-to-one, let $\phi(i) = \phi(i')$.  Then
\[
\alpha_k(i') b^{c r_k'}
= b^d \alpha_k(\phi(i'))
= b^d \alpha_k(\phi(i))
= \alpha_k(i) b^{c r_k}.
\]
By the uniqueness of form (L) we have that $i' = i$.  From the definition of $\phi$ we see that $\phi(i)_k$ is locally constant, so that $\phi$ is continuous, and hence a homeomorphism.
\end{proof}

\begin{Lemma}
\label{l.casethree}
Assume we are in case (BS3).  There are maps $\phi : [0,d)^\IN \to [0,d)^\IN$ and $r : [0,d)^\IN \to \IN^\IN$ such that for all $k \ge 0$ we have $b^d \alpha_k(\phi(i)) = \alpha_k(i) b^{(-1)^k c r(i)_k}$.  Moreover, $\phi$ is a homeomorphism of $[0,d)^\IN$ for the product topology.
\end{Lemma}

\begin{proof}
For $t \in \IR$ let $f_0(t) = \lfloor t \rfloor$ and $f_1(t) = \lceil t \rceil$.  Let $\tau(k) = \tfrac{1}{2} \bigl( 1 - (-1)^k \bigr)$.  We define $r$ and $\phi$ recursively by $r(i)_0 = 1$, $\phi(i)_0 = i_0$ and 
\begin{align*}
r(i)_k &= f_{\tau(k)} \left( \frac { (-1)^k i_k + r(i)_{k-1} c} {d} \right) \\
\phi(i)_k &= i_k + (-1)^k \bigl( r(i)_{k-1} c - r(i)_k d \bigr).
\end{align*}
We first show that $r(i)_k \ge 0$ for all $k$.  This is true for $k = 0$; suppose it is true for $k-1$.  If $k$ is even, then $r(i)_k = \bigl\lfloor \tfrac{1}{d} \bigl( i_k + r(i)_{k-1} c \bigr) \bigr\rfloor \ge \lfloor 0 \rfloor = 0$.  If $k$ is odd, $r(i)_k = \bigl\lceil \tfrac{1}{d} \bigl( -i_k + r(i)_{k-1} c \bigr) \bigr\rceil \ge \bigl\lceil -\tfrac {i_k} {d} \bigr \rceil \ge 0$, since $-\tfrac {i_k} {d} > -1$.

We next show that $\phi(i)_k \in [0,d)$ for all $k$.  Again, it is true for $k = 0$.  Suppose it is true for $k - 1$.  If $k$ is even, $r(i)_k = \bigl\lfloor \tfrac{1}{d} \bigl( i_k + r(i)_{k-1} c \bigr) \bigr\rfloor$, hence $r(i_k) \le \tfrac{1}{d} \bigl( i_k + r(i)_{k-1} c < r(i)_k + 1$, and hence $0 \le i_k + r(i)_{k-1}c - r(i)_k d < d$, which is the statement that $\phi(i)_k \in [0,d)$.  If $k$ is odd, $r(i)_k = \bigl\lceil \tfrac{1}{d} \bigl( -i_k + r(i)_{k-1} c \bigr) \bigr\rceil$, hence $r(i)_k - 1 < \tfrac{1}{d} \bigl( -i_k + r(i)_{k-1} c \bigr) \le r(i)_k$, and hence $0 \le i_k - \bigl( r(i)_{k-1} c - r(i)_k d \bigr) < d$, which is the statement that $\phi(i)_k \in [0,d)$.

Now we check the equation of the statement.  It is true for $k = 0$.  Suppose it is true for $k-1$.  Then
\[
b^d \alpha_k(\phi(i))
= b^d \alpha_{k-1}(\phi(i)) b^{\phi(i)_k} a
= \alpha_{k-1}(i) b^{(-1)^{k-1} c r(i)_{k-1}} b^{\phi(i)_k} a. 
\]
Note that $(-1)^{k-1} c r(i)_{k-1} + \phi(i)_k = i_k - (-1)^k r(i)_k d$.  Thus
\[
b^d \alpha_k(\phi(i))
= \alpha_{k-1}(i) b^{i_k - (-1)^k r(i)_k d} a
= \alpha_{k-1}(i) b^{i_k} a b^{(-1)^k c r(i)_k}
= \alpha_k(i) b^{(-1)^k c r(i)_k}.
\]
From the definition of $\phi$ we see that $\phi(i)_k$ is locally constant, so that $\phi$ is continuous.  We show that $\phi$ is bijective, and hence is a homeomorphism.  For this, we construct the inverse of $\phi$.  We define maps $\psi : [0,d)^\IN \to [0,d)^\IN$ and $s : [0,d)^\IN \to \{0,1\}^\IN$ recursively by $s(i)_0 = 1$, $\psi(i)_0 = i_0$ and 
\begin{align*}
s(i)_k &= f_{\tau(k-1)} \left( \frac { (-1)^{k-1} i_k + s(i)_{k-1} c} {d} \right) \\
\psi(i)_k &= i_k + (-1)^{k-1} \bigl( s(i)_{k-1} c - s(i)_k d \bigr).
\end{align*}
The proof that $s(i)_k \ge 0$ and $\psi(i)_k \in [0,d)$ for all $k$ is similar to the proof of the analogous facts for $r$ and $\phi$ above.    We show that $r \circ \psi = s$ and $s \circ \phi = r$.  We have
\begin{align*}
r (\psi (i))_k
&= f_{\tau(k)} \left( \frac {(-1)^k \psi(i)_k + r(\psi(i))_{k-1} c} {d} \right) \\
&= f_{\tau(k)} \left( \frac {(-1)^k \bigl(  i_k + (-1)^{k-1} s(i)_{k-1} c + (-1)^k s(i)_k d \bigr) + s(i)_{k-1} c} {d} \right) \\
&= f_{\tau(k)} \left( \frac {(-1)^k i_k} {d} \right) + s(i)_k \\
&= s(i)_k; \\
s(\phi(i))_k
&=  f_{\tau(k-1)} \left( \frac { (-1)^{k-1} \phi(i)_k + s(\phi(i))_{k-1} c} {d} \right) \\
&=  f_{\tau(k-1)} \left( \frac { (-1)^{k-1} \bigl(  i_k + (-1)^k r(i)_{k-1} c -(-1)^k r(i)_k d \bigr) + ri())_{k-1} c} {d} \right) \\
&=  f_{\tau(k-1)} \left( \frac { (-1)^{k-1} i_k} {d} \right) + r(i)_k \\
&= r(i)_k,
\end{align*}
since $f_{\tau(n)} \bigl( (-1)^n t \bigr) = 0$ for $0 \le t < 1$.  Now we have 
\[
\psi(\phi(i))_k
= \phi(i)_k + (-1)^{n-1} \bigl( s(\phi(i))_{k-1} c - s(\phi(i))_k d \bigr)
= \phi(i)_k + (-1)^{k-1} \bigl( r(i)_{k-1} c - r(i)_k d \bigr) 
= i_k,
\]
and similarly, $\phi(\psi(i))_k = i_k$.
\end{proof}

\section{Directed hereditary subsets of $\Lambda$}
\label{s.directedhereditarysets}

We pause a moment to describe the standard picture of $\Lambda$.  This is just the portion of the Cayley graph of $G$ given by $\Lambda$, with direction given by the generators.  The set $\Lambda / B$ then becomes a directed tree --- the portion of the Bass-Serre tree for $G$ corresponding to $\Lambda$.  We think of the Cayley graph as made up of branching \textit{sheets}, which we will describe precisely in a moment.  The tree is then the side view, where each sheet becomes an infinite path.  (See e.g. \cite{farmos} for a sketch of the case (BS2), $c = 1$, $d = 2$.)  The precise description of these objects is found in the identification of the directed hereditary subsets of $\Lambda$, i.e. of $\Lambda^*$.

For $\alpha \in \Lambda$ we have the directed hereditary subset $[\alpha]$.  Moreover, since $\alpha\Lambda \setminus \alpha a \Lambda \setminus \alpha b \Lambda = \{ \alpha \}$, it follows that $[\alpha]$ is an open point of $\Lambda^*$.  Thus $\Lambda$ is a discrete subset of $\Lambda^*$ (where we identify $\Lambda$ with the set of finite directed hereditary subsets of $\Lambda$).

Next, for $\alpha B \in \Lambda / B$ we let $[\alpha B] = \bigcup_{n=0}^\infty [\alpha b^n]$.  It is clear that $[\alpha B]$ is hereditary and directed, hence is an element of $\Lambda^*$.  Letting $\alpha = b^{i_0} a \cdots b^{i_k} a$ in form (L), it follows from Corollary \ref{c.initseq} that $\alpha$ is uniquely determined by $[\alpha B]$.  Thus we may identify $\{ [\alpha B] : \alpha \in \Lambda \}$ with $\Lambda / B$. For $m \ge 0$ set $U_m = \alpha b^m \Lambda \setminus \bigcup_{\mu = 0}^{d-1} \alpha b^\mu a \Lambda$.  Then $U_m = \{ \alpha b^j : j \ge m \}$ defines a neighborhood $\widehat{U_m}$ of $[\alpha B]$ in $\Lambda^*$.  Thus $\Lambda / B$ is a relatively discrete subset of $\Lambda^* \setminus \Lambda$.

\begin{Definition}
\label{d.infiniteheight}
An element $C \in \Lambda^*$ is of \textit{finite height} if $\theta$ is bounded on $C$.  An element of $\Lambda^*$ not of finite height is said to be of \textit{infinite height}.
\end{Definition}

\begin{Lemma}
\label{l.finiteheight}
The elements of $\Lambda^*$ of finite height are $\{ [\alpha] : \alpha \in \Lambda \} \cup \{ [\xi] : \xi \in \Lambda / B \}$.
\end{Lemma}

\begin{proof}
Let $C \in \Lambda^*$ have finite height.  Let $\alpha \in C$ have maximal height.  By the hereditary property of $C$ we may assume that $\alpha = b^{i_0} a \cdots b^{i_k} a$ in form (L).  By Corollary \ref{c.initseq}, every element of $C$ of height $k$ is of the form $\alpha b^p$ for some $p$.  If $\{ p : \alpha b^p \in C \}$ has a maximum element $q$, then $C = [\alpha b^q]$.  Otherwise, we have $C = [\alpha B]$.
\end{proof}

\begin{Remark}
The elements $\Lambda / B$ form a directed tree when ordered by containment; this follows from Corollary \ref{c.initseq}.  For $\alpha = b^{i_0} a \cdots b^{i_k} a$, the immediate successors of $[\alpha B]$ in this tree are $[\alpha a B]$, $[\alpha b a B]$, $\ldots$, $[\alpha b^{d-1} a B]$.  Thus if we orient the tree so as to be directed upward, each vertex has $d$ upward edges (and, apart from $B$, one downward edge).
\end{Remark}

We now consider the the directed hereditary subsets having infinite height.

\begin{Definition}
\label{d.infht}
Let $i \in [0,d)^\IN$.  
\begin{align*}
C_0(i) &= \bigcup_{k = 0}^\infty [\alpha_k(i)] \\
C_\infty(i) &= \bigcup_{k,p = 0}^\infty [\alpha_k(i) b^p].
\end{align*}
\end{Definition}

\begin{Lemma}
\label{l.minmax}
\begin{enumerate}
\item \label{l.minmax.a} $C_0(i)$ and $C_\infty(i)$ are directed hereditary subsets of $\Lambda$.
\item \label{l.minmax.b} $C_0(i) \subseteq C_\infty(i)$.
\item \label{l.minmax.c} $\bigl\{ C_\infty(i) : i \in [0,d)^\IN \bigr\} = \Lambda^{**}$.
\item \label{l.minmax.d} If $C \in \Lambda^*$ is of infinite height, there is a unique $i \in [0,d)^\IN$ such that $C \subseteq C_\infty(i)$.  Moreover $C_0(i) \subseteq C$.
\end{enumerate}
\end{Lemma}

\begin{proof}
\eqref{l.minmax.a} $C_0(i)$ is directed since it is an increasing union of directed sets.  $C_\infty(i)$ is directed by Lemmas \ref{l.clessthand}\eqref{l.clessthand.d} and \ref{l.cged}\eqref{l.cged.d}.  Both are hereditary since they are unions of hereditary sets.

\eqref{l.minmax.b} is immediate.

\eqref{l.minmax.c} To see that $C_\infty(i)$ is a maximal directed hereditary subset, let $\beta \in \Lambda \setminus C_\infty(i)$.  Write $\beta = b^{j_0} a \cdots b^{j_k} a b^p$ in form (L).  Then there is $\ell \le k$ such that $j_\ell \not= i_\ell$.  By Proposition \ref{p.qlo} we have that $\beta \perp \alpha_k(i)$, and hence there cannot exist a directed hereditary subset containing $\beta$ and $C_\infty(i)$.  Now we show that these are all of the maximal elements of $\Lambda^{**}$.  Let $C \in \Lambda^{*}$.  Since any two elements of $C$ have a common extension, Proposition \ref{p.qlo} implies that there is $i \in [0,d)^\IN$ such that $C \subseteq \bigcup_{k,p=0}^\infty [\alpha_k(i) b^p] = C_\infty(i)$.

\eqref{l.minmax.d} If $C \in \Lambda^*$ is of infinite height, the sequence $i$ in the proof of part \eqref{l.minmax.c} is uniquely determined by $C$.  For each $k$ there is $p \ge 0$ such that $\alpha_k(i) b^p \in C$.  Therefore $\alpha_k(i) \in C$ for all $k$, and hence $C_0(i) \subseteq C$.
\end{proof}

We now describe precisely the sheets mentioned above.  Namely, the \textit{sheets} are the subsets $C_\infty(i)$ of $\Lambda$, for $i \in [0,d)^\IN$.  The cosets $\bigl\{ \alpha_k(i) B : k \in \IN \bigr\}$ define an infinite path in the tree $\Lambda / B$.  Thus we may identify the boundary of $\Lambda$ with the boundary of the directed tree $\Lambda / B$.

\begin{Lemma}
\label{l.containsb}
Let $i \in [0,d)^\IN$ and let $C \in \Lambda^*$ with $C_0(i) \subseteq C$.  Then $C = C_\infty(i)$ if and only if $B \subseteq C$.
\end{Lemma}

\begin{proof}
Of course, the hypothesis $C_0(i) \subseteq C$ is equivalent to assuming that $C \subseteq C_\infty(i)$ and that $C$ is of infinite height.  We first consider case (BS1).  By Lemma \ref{l.cged}\eqref{l.cged.d} we have
\[
B \vee C_0(i)
\supseteq \bigl[ b^m \vee \alpha_k(i) \bigr]
= \bigl[ \alpha_k(i) b^{c \sum_{\mu = 0} ^{h -1} r(\phi^\mu(i))_k} \bigr].
\]
Thus $B \vee C_0(i) \supseteq C_\infty(i)$.  Conversely, given $h \ge 0$ let $m = i_0 + hd$.  Then
\[
b^m
\in \bigl[ b^m \vee \alpha_k(i) \bigr]
= \bigl[ \alpha_k(i) b^{c \sum_{\mu = 0} ^{h -1} r(\phi^\mu(i))_k} \bigr] 
\subseteq C_\infty(i).
\]
Hence $B \subseteq C_\infty(i)$.

Now we consider case (BS2).  Fix $k$.  There are $i' \in [0,c)^k$ and $q \ge 0$ such that $\alpha_k(i) = b^q \alpha(i')$ in form (R).  Let $h \ge 0$, and choose $m = q + d \sum_{\mu = 0}^h r(\psi^\mu(i))$ (where $\psi$ is as in Lemma \ref{l.clessthand}).  By Lemma \ref{l.clessthand}\eqref{l.clessthand.d} we have
\[
b^m \vee \alpha_k(i)
= b^q \bigl( b^{m - q} \vee \alpha(i') \bigr)
= b^q \alpha(i') b^{(h + 1)c}
= \alpha_k(i) b^{(h+1)c}.
\]
If $B \subseteq C$, then since $h$ was arbitrary, we see that $C_\infty(i) \subseteq B \vee C_0(i) \subseteq C$.  Conversely, for all $h$ we have
\[
b^{hd} \in \bigl[ b^{d \sum_{\mu = 0}^{h-1} r(\psi^\mu(i))} \alpha(\psi^h(i)) \bigr]
= \bigl[ \alpha(i) b^{hc} \bigr]
\subseteq C_\infty(i).
\]
Thus $B \subseteq C_\infty(i)$.

Finally we consider case (BS3).  By Lemma \ref{l.bincaseethree}, $[\alpha_{k+1}(i)] = [\alpha_k(i) b^{i_{k+1}} a] \supseteq [\alpha_k(i) B]$.  Thus $C_0(i) = C_\infty(i)$, and of course, $B \subseteq C_\infty(i)$.
\end{proof}

\begin{Corollary}
\label{c.closedboundary}
$\Lambda^{**}$ is a closed subset of $\Lambda^*$ (and thus $\partial \Lambda = \Lambda^{**}$).  The map $C_\infty(i) \mapsto i$ of $\partial \Lambda \to [0,d)^\IN$ is a homeomorphism, equivariant for the maps $b^d \cdot$ on $\partial \Lambda$, and $\phi^{-1}$ on $[0,d)^\IN$ (from Lemmas \ref{l.cged}, \ref{l.cltd}, and \ref{l.casethree}).
\end{Corollary}

\begin{proof}
It was pointed out in the remarks before Definition \ref{d.infiniteheight} that the subset of $\Lambda^*$ consisting of elements of finite height is an open subset of $\Lambda^*$.  Thus we must show that if $C \not\in \Lambda^{**}$ is of infinite height, then $C$ has a neighborhood disjoint from $\Lambda^{**}$.  By Lemma \ref{l.containsb} we know that $m = \max \{ \ell : b^\ell \in C \}$ is finite.  Then $\Lambda^* \setminus \widehat{b^{m + 1}\Lambda}$ is a neighborhood of $C$ disjoint from $\Lambda^{**}$.

It is clear that $C_\infty(i) \mapsto i$ is a bijection of $\partial \Lambda$ with $[0,d)^\IN$.  Given $(j_0, \ldots, j_k) \in [0,d)^{k+1}$, the open set $(b^{j_0} a \cdots b^{j_k} a \Lambda)^{\widehat{\null}}$ corresponds to the cylinder set defined by $(j_0, \ldots, j_k)$, showing that the map is a homeomorphism.  The equivariance is clear from the three lemmas mentioned.
\end{proof}

\begin{Remark}
\label{r.odometer}
It follows from Corollary \ref{c.closedboundary} that left-concatenation by $b$ defines a homeomorphism of $\partial \Lambda$.
\end{Remark}

\begin{Definition}
\label{d.sigma}
Let $\Sigma \equiv \Sigma(\Lambda) = \Lambda^* \setminus \bigl( \Lambda \cup (\Lambda / B) \cup \Lambda^{**} \bigr)$.
\end{Definition}

\begin{Lemma}
\label{l.sigmaopen}
$\Sigma$ is a relatively open subset of $\Lambda^* \setminus \Lambda$.
\end{Lemma}

\begin{proof}
Let $C \in \Sigma$.  Then $m = \sup \{ j : b^j \in C \} < \infty$.  Then $\Lambda^* \setminus \widehat{b^{m+1}\Lambda}$ is a neighborhood of $C$ disjoint from $(\Lambda / B) \cup \Lambda^{**}$.
\end{proof}

We now see that $G^0 = \Lambda^*$ is the disjoint union of four invariant subsets:  $\Lambda^* = \Lambda \sqcup (\Lambda / B) \sqcup \partial \Lambda \sqcup \Sigma$, where the meanings of the first two subsets were specified in the remarks before Definition \ref{d.infiniteheight} (in case (BS3), the set $\Sigma$ is empty).  The subset $\Lambda$ is discrete.  It was shown in the remarks before Definition \ref{d.infiniteheight} that the subset $\Lambda / B$ is relatively discrete, hence open, in $\Lambda^* \setminus \Lambda$.  We now have seen that $\Sigma$ is also open in $\Lambda^* \setminus \Lambda$.  

Let $G = G(\Lambda)$ be the groupoid of $\Lambda$.  We wish to prove that $G$ is amenable.  We note that since the maximal degree functor $\theta$ is degenerate,  the fixed-point groupoid is not AF.  However, we may still use Proposition \ref{p.nuclearity}.  The first step is to show that the restriction of $G$ to $\partial \Lambda$ is amenable.

\begin{Theorem}
\label{t.boundaryamenable}
$G |_{\partial \Lambda}$ is amenable.
\end{Theorem}

\begin{proof}
By Proposition \ref{p.nuclearity}, it suffices to show that $G^\theta |_{\partial \Lambda}$ is amenable.  Let $H = G^\theta |_{\partial \Lambda}$.  Thus $H = \bigl\{ [\alpha,\beta,x] : \theta(\alpha) = \theta(\beta),\ x \in \partial \Lambda \bigr\}$.  For $k \in \IN$ let $H_k = \bigl\{ [\alpha,\beta,x] \in H : \theta(\alpha) = \theta(\beta) = k \bigr\}$.  Then $H = \bigcup_{k = 0}^\infty H_k$.  First we show that $H_k$ is a subgroupoid of $H$.  It is clear that $H_k$ is closed under inversion.  Let $\bigl( [\alpha,\beta,x],[\gamma,\delta,y] \bigr) \in H_k^{(2)}$.  Since $\beta x = \gamma y$, we have that in form (L), $\beta = b^{i_0} a \cdots \beta^{i_k} a b^p$ and $\gamma = b^{i_0} a \cdots b^{i_k} a b^q$.  Thus $y = b^{p - q} x$ (cf. Remark \ref{r.odometer}).  Therefore, letting $\eps = b^{i_0} a \cdots b^{i_k} a$, we have
\[
[\alpha,\beta,x] [\gamma,\delta,y] 
= [\alpha,\eps b^p x] [\eps b^q,\delta,b^{p - q}x] 
= [\alpha b^q,\eps b^{p + q}, b^{-q} x] [\eps b^{p + q},\delta b^p, b^{-q}x] 
= [\alpha b^q,\delta b^p, b^{-q}x] 
\in H_k.
\]
Next we show that $H_k \subseteq H_{k + 1}$.  Let $\theta(\alpha) = \theta(\beta) = k$, and $[\alpha,\beta,x] \in H_k$.  Write $x = C_\infty(i)$ and $y = C_\infty(\sigma(i))$, where we use $\sigma$ to denote the (noninvertible) left shift on $[0,d)^\IN$.  Then
\[
[\alpha,\beta,x]
= [\alpha b^{i_0} a, \beta b^{i_0} a, y]
\in H_{k + 1}.
\]
Note also that
\[
H_k 
= \bigcup_{i,j \in [0,d)^k} \bigcup_{p,q \in \IN} 
[b^{i_0} a \cdots b^{i_{k - 1}} a b^p, b^{j_0} a \cdots b^{j_{k - 1}} a b^q, \partial \Lambda]
\]
is an open subgroupoid of $H$, hence also of $H_{k + 1}$.

Now we observe that the map $[b^p, b^q, x] \mapsto (p - q, x)$ is an isomorphism of $H_0$ onto the transformation groupoid $\IZ \ltimes \partial \Lambda$.  Thus $H_0$ is amenable.  Finally, if we write the multiplication in $H_k$ as
\[
[b^{i_1} a \cdots b^{i_k} a b^p, b^{j_1} a \cdots b^{j_k} a b^q, x]
[ b^{j_1} a \cdots b^{j_k} a b^q,  b^{\ell_1} a \cdots b^{\ell_k} a b^r, x]
= [b^{i_1} a \cdots b^{i_k} a b^p, b^{\ell_1} a \cdots b^{\ell_k} a b^r, x],
\] 
we see that $H_k$ is isomorphic to the product groupoid
$\bigl( [0,d)^k \times [0,d)^k \bigr) \times H_0$, which is amenable.  Therefore $H$ is amenable (by \cite{ren}, III.1).
\end{proof}

Before proving that $G$ is amenable, we require a detailed description of the remaining elements of $\Lambda^*$ having infinite height, that is, the elements of $\Sigma$.  We first consider case (BS1).  Let $i \in [0,d)^\IN$.  Let $n_0$, $n_1$, $\ldots \in \IN$ satisfy
\begin{equation}
\label{e.inequalities}
\frac{c n_{\ell - 1} - i_\ell}{d}
\le n_\ell
< \frac{c(n_{\ell - 1} + 1) - i_\ell}{d}, \quad \ell \ge 1. \tag{$**$}
\end{equation}
\vspace*{.05 in}

Note that since the outside terms of these inequalities differ by $\tfrac{c}{d} \ge 1$, such sequences exist for any choice of $n_0$.  We let $n = (n_0, n_1, \ldots)$, and set
\[
C_n(i) = \bigcup_{\ell = 0}^\infty [\alpha_{\ell - 1}(i) b^{n_\ell d + i_\ell}].
\]

\begin{Lemma}
\label{l.sequencea}
The sets in the above union increase, and hence $C_n(i)$ is directed and hereditary.
\end{Lemma}

\begin{proof}
Since $c n_\ell \le n_{\ell + 1} d + i_{\ell + 1}$, we have
\[
\alpha_{\ell - 1}(i) b^{n_\ell d + i_\ell}
\in [\alpha_{\ell - 1}(i) b^{n_\ell d + i_\ell} a]
= [\alpha_\ell(i) b^{n_\ell c}]
\subseteq [\alpha_\ell(i) b^{n_{\ell + 1} d + i_{\ell + 1}}]. \qedhere
\]
\end{proof}

We remark that what we have denoted $C_0(i)$ equals the directed hereditary set defined as above for the sequence $n$ consisting entirely of zeros.

\begin{Lemma}
\label{l.sequenceb}
$\alpha_{\ell - 1}(i) b^{(n_\ell + 1)d} \not\in C_n(i)$.
\end{Lemma}

\begin{proof}
We will show that for $m \ge \ell$ we have $\alpha_{\ell - 1}(i) b^{(n_\ell + 1)d} \not\in [\alpha_{m - 1}(i) b^{n_m d + i_m}]$.  Since $i_\ell < d$, this is true when $m = \ell$.  Let $m \ge \ell$, and suppose inductively that 
$b^{i_\ell} a \cdots b^{i_{m - 1}} a b^{n_m d + i_m} 
= b^{n_\ell d + i_\ell} a b^{j_{\ell + 1}} \cdots a b^{j_m}$,
where $j_{\ell + 1}$, $\ldots$, $j_m \in [0,c)$ (this is true vacuously when $m = \ell$).  Since
$c n_m \le d n_{m + 1} + i_{m + 1} < c(n_m + 1)$,
we have that
$a b^{n_{m + 1} d + i_{m + 1}} = b^{n_m d} a b^{j_{m + 1}}$,
where $j_{m + 1} \in [0,c)$.  Then
\[
b^{i_\ell} a \cdots b^{i_m} a b^{n_{m + 1} d + i_{m + 1}}
= b^{i_\ell} a \cdots b^{i_{m - 1}} a b^{n_m d + i_m} a b^{j_{m + 1}} 
= b^{n_\ell d + i_\ell} a b^{j_{\ell + 1}} \cdots a b^{j_{m + 1}}.
\]
Since this is in form (R), we see that $b^{(n_\ell + 1)d} \not\in [b^{i_\ell} a \cdots b^{n_{m + 1} d + i_{m + 1}}]$, since $(n_\ell + 1)d > n_\ell d + i_\ell$.  Therefore $\alpha_{\ell - 1}(i) b^{(n_\ell + 1)d} \not\in [\alpha_m(i) b^{n_{m + 1} d + i_{m + 1}}]$.
\end{proof}

\begin{Corollary}
\label{c.sequencec}
$n_\ell = \max \bigl\{ m : \alpha_{\ell - 1}(i) b^{md} \in C_n(i) \bigr\}$.
\end{Corollary}

\begin{Lemma}
\label{l.sequenced}
Let $n$ and $n'$ both satisfy the inequalities \eqref{e.inequalities}. Suppose that $n_k = n'_k$ for $k < \ell$, and that $n_\ell < n'_\ell$.  Then $C_n(i) \subsetneq C_{n'}(i)$.
\end{Lemma}

\begin{proof}
By Corollary \ref{c.sequencec} we know that $\alpha_{\ell - 1}(i) b^{n'_\ell d} \in C_{n'}(i) \setminus C_n(i)$, and hence $C_n(i) \not= C_{n'}(i)$. Since $n'_\ell \ge n_\ell + 1$, we have
\[
n_{\ell + 1}
< \frac{c (n_\ell + 1) - i_{\ell + 1}}{d}
\le \frac{c n'_\ell - i_{\ell + 1}}{d}
\le n'_{\ell + 1}.
\]
Inductively we find that $n_k < n'_k$ for $k \ge \ell$, and hence that $C_n(i) \subseteq C_{n'}(i)$.
\end{proof}

\begin{Lemma}
\label{l.sequencee}
Let $C \in \Lambda^*$ with $C_0(i) \subseteq C \subsetneq C_\infty(i)$.  Then there exists $n = (n_0, n_1, \ldots)$ satisfying the inequalities \eqref{e.inequalities} such that $C = C_n(i)$.
\end{Lemma}

\begin{proof}
Since $C \not= C_\infty(i)$, Lemma \ref{l.containsb} implies that $B \not\subseteq C$.  Thus we may define $n_\ell = \max \bigl\{m : \alpha_{\ell - 1}(i) b^{md} \in C \bigr\}$.  Thus $\alpha_{\ell - 1}(i) b^{n_\ell d} \in C$ and $\alpha_{\ell - 1}(i) b^{(n_\ell + 1)d} \not\in C$.  Let $C' = \sigma^{\alpha_{\ell - 1}(i)} C$.  Then $b^{n_\ell d} \in C'$ and $b^{i_\ell} a \in C'$.  Therefore
\[
b^{n_\ell d + i_\ell} a = b^{n_\ell d} \vee b^{i_\ell} a \in C'.
\]
Thus $b^{n_\ell d + i_\ell} \in C'$, and hence $\alpha_{\ell - 1}(i) b^{n_\ell d + i_\ell} \in C$.

We claim that the sequence $n = (n_0,n_1,\ldots)$ satisfies the inequalities \eqref{e.inequalities}.  For the strict inequality, suppose otherwise; i.e. suppose that $i_\ell + d n_\ell \ge c(n_{\ell - 1} + 1)$.  Then $\alpha_{\ell - 1}(i) b^{c(n_{\ell - 1} + 1)} \in C$, and hence $\alpha_{\ell - 2}(i) b^{i_{\ell - 1} + d(n_{\ell - 1} + 1)} \in C$.  But this contradicts the definition of $n_{\ell - 1}$.

For the weak inequality, we already know that $\alpha_{\ell - 2}(i) b^{i_{\ell - 1} + d n_{\ell - 1}} \in C$.  Thus $b^{i_{\ell - 1} + d n_{\ell - 1}}$, $b^{i_{\ell - 1}} a \in \sigma^{\alpha_{\ell - 2}(i)} C$.  Therefore 
\[
b^{i_{\ell - 1}} a b^{c n_{\ell - 1}}
= b^{i_{\ell - 1} + d n_{\ell - 1}} \vee b^{i_{\ell - 1}} a
\in \sigma^{\alpha_{\ell - 2}(i)} C.
\]
Hence $\alpha_{\ell - 1}(i) b^{c n_{\ell - 1}} \in C$.

On the other hand, we know that $\alpha_{\ell - 1}(i) b^{i_\ell + n_\ell d} \in C$.  Moreover, by Lemma \ref{l.cged}\eqref{l.cged.d}, we have that
\[
b^{n_\ell d + i_\ell + 1} \vee b^{i_\ell} a = b^{i_\ell} a b^{c(n_\ell + 1)}.
\]
Hence if $\alpha_{\ell - 1}(i) b^{n_\ell d + i_\ell + 1} \in C$, then
$\alpha_{\ell - 1}(i) b^{(n_\ell + 1)d + i_\ell} a
= \alpha_\ell(i) b^{c(n_\ell + 1)} \in C$,
contradicting the definition of $n_\ell$.  Therefore $n_\ell d + i_\ell = \max \bigl\{ m : \alpha_{\ell - 1}(i) b^m \in C \bigr\}$.  Thus we conclude that $c n_{\ell - 1} \le n_\ell d + i_\ell$.  This finishes the demonstration that $n$ satisfies the inequalities \eqref{e.inequalities}.  By its definition, we have $C_n(i) \subseteq C$.

Finally, we claim that $C_n(i) = C$.  For this consider a typical element $\beta \in C$.  Then $\beta = \alpha_{\ell - 1}(i) b^q$ for some $\ell$ and $q$.  Then from the above argument we must have $q \le n_\ell d + i_\ell$, and hence that $\beta \in C_n(i)$.
\end{proof}

\begin{Corollary}
\label{c.cgedsigma}
$\Sigma = \{ C_n(i) : i \in [0,d)^\IN,\ n$ satisfies the inequalities \eqref{e.inequalities} for $i \}$.  
\end{Corollary}

We will now study the restriction of $G$ to $\Sigma$.

\begin{Lemma}
\label{l.bdonsigma}
Recall the maps $\phi$ and $r$ from Lemma \ref{l.cged}.  Let $i \in [0,d)^\IN$.
\begin{enumerate}
\item \label{l.bdonsigmaone}
Let $n$ satisfy the inequalities \eqref{e.inequalities} for $\phi(i)$.  Set $n' = n + r(i)$.  Then $n'$ satisfies the inequalities \eqref{e.inequalities} for $i$.  Moreover $b^d C_n(\phi(i)) = C_{n'}(i)$.
\item \label{l.bdonsigmatwo}
Let $n'$ satisfy the inequalities \eqref{e.inequalities} for $i$, and suppose that $n'_0 \ge 1$.  Then $n' \ge r(i)$, and $n = n' - r(i)$ satisfies the inequalities \eqref{e.inequalities} for $\phi(i)$.
\end{enumerate}
\end{Lemma}

\begin{proof}
\eqref{l.bdonsigmaone}:  $n$ satisfies the inequalities \eqref{e.inequalities} for $\phi(i)$ if and only if (for all $\ell$,)
\[
c n_{\ell - 1} \le d n_\ell + \phi(i)_\ell \le c n_{\ell - 1} + c.
\]
Using equation \eqref{e.phiandr} from Lemma \ref{l.cged}, this is equivalent to
\[
c n_{\ell - 1} \le d n_\ell + i_\ell + d r(i)_\ell - c r(i)_{\ell - 1} \le c n_{\ell - 1} + c,
\]
and hence to
\[
c n'_{\ell - 1} \le d n'_\ell + i_\ell \le c n'_{\ell - 1} + c.
\]
Then we have
\begin{align*}
b^d C_n(\phi(i))
&= \bigcup_{\ell = 0}^\infty \bigl[ b^d \alpha_{\ell - 1}(\phi(i)) b^{n_\ell d + \phi(i)_\ell} \bigr] 
= \bigcup_{\ell = 0}^\infty \bigl[ \alpha_{\ell - 1}(i) b^{n_\ell d + \phi(i)_\ell + c r(i)_{\ell - 1}} \bigr], \text{ by Lemma \ref{l.cged}\eqref{l.cged.a},} \\
&= \bigcup_{\ell = 0}^\infty \bigl[ \alpha_{\ell - 1}(i) b^{n_\ell d + i_\ell + d r(i)_\ell} \bigr], \text{ by equation \eqref{e.phiandr},} \\
&= \bigcup_{\ell = 0}^\infty \bigl[ \alpha_{\ell - 1}(i) b^{n'_\ell d + i_\ell} \bigr] 
= C_{n'}(i).
\end{align*}

\eqref{l.bdonsigmatwo}:  We show that $n'_\ell \ge r(i)_\ell$ for all $\ell \in \IN$.  By assumption this holds for $\ell = 0$.  Suppose it holds for $\ell - 1$.  Then using the inequalities \eqref{e.inequalities} for $n'$, and equation \eqref{e.phiandr}, we have
\[
d n'_\ell
\ge c r(i)_{\ell - 1} - i_\ell
= d r(i)_\ell - \phi(i)_\ell
> d r(i)_\ell - d.
\]
Therefore $n'_\ell > r(i)_{\ell} - 1$, and hence $n'_\ell \ge r(i)_\ell$.  Thus $n = n' - r(i) \ge 0$.  The proof of \eqref{l.bdonsigmaone} shows that $n$ satisfies the inequalities \eqref{e.inequalities} for $\phi(i)$.
\end{proof}

\begin{Theorem}
\label{t.sigmaamenable}
Assume that we are in case (BS1).  Then  $G |_\Sigma$ is amenable.
\end{Theorem}

\begin{proof}
Since the action of $b$ on $\Sigma$ is not surjective, the argument differs in a few places from that of Theorem \ref{t.boundaryamenable}.  By Proposition \ref{p.nuclearity}, it suffices to show that $G^\theta |_\Sigma$ is amenable.  Let $M = G^\theta |_\Sigma = \bigl\{ [\beta,\gamma,x] : x \in \Sigma,\ \theta(\beta) = \theta(\gamma) \bigr\}$.  Write $M = \bigcup_{k = 0}^\infty M_k$, where $M_k = \bigl\{ [\beta,\gamma,x] \in M : \theta(\beta) = \theta(\gamma) = k \bigr\}$.  Letting $\sigma$ denote the left shift on sequences (as well as the left shift on $\Lambda$), we have
\begin{align*}
\sigma^{b^{i_0} a} C_n(i)
&= \sigma^{b^{i_0} a} \bigcup_{\ell = 0}^\infty \bigl[ \alpha_{\ell - 1}(i) b^{n_\ell d + i_\ell} \bigr] 
= \bigcup_{\ell = 0}^\infty \bigl[ b^{i_1} a \cdots b^{i_\ell} a b^{n_{\ell + 1} d + i_{\ell + 1}} \bigr] \\
&= \bigcup_{\ell = 0}^\infty \bigl[ \alpha_{\ell - 1} (\sigma(i)) b^{\sigma(n)_\ell d + \sigma(i)_\ell} \bigr] 
= C_{\sigma(n)}(\sigma(i)).
\end{align*}
Thus if $\theta(\beta) = \theta(\gamma) = k$,
\[
[\beta,\gamma,C_n(i)]
= \bigl[ \beta, \gamma, b^{i_0} a C_{\sigma(n)}(\sigma(i)) \bigr]
= \bigl[ \beta b^{i_0} a, \gamma b^{i_0} a, C_{\sigma(n)}(\sigma(i)) \bigr]
\in M_{k + 1}.
\]
It is clear that $M_k$ is closed under inversion.  If $[\alpha,\beta,C_m(i)]$ and $[\gamma,\delta, C_n(j)]$ are composable elements of $M_k$, then $\beta C_m(i) = \gamma C_n(j)$ and $\theta(\beta) = \theta(\gamma) = k$.  Then in form (L) we have $\beta = b^{\mu_0} a \cdots b^{\mu_{k - 1}} a b^p$ and $\gamma = b^{\mu_0} a \cdots b^{\mu_{k - 1}} a b^q$, with $p$, $q \ge 0$.  Without loss of generality, assume that $p \le q$.  Then $\gamma = \beta b^{q - p}$, and $C_m(i) = b^{q - p} C_n(j)$.  Then
\[
[\alpha,\beta,C_m(i)] \, [\gamma,\delta, C_n(j)]
= [\alpha,\beta, b^{q - p} C_n(j)] \, [\beta b^{q - p},\delta, C_n(j)]
= [\alpha b^{q - p}, \delta, C_n(j)]
\in M_k.
\]
In the same way as in the proof of Theorem \ref{t.boundaryamenable} we have that $M_k$ is isomorphic to $\bigl( [0,d)^k \times [0,d)^k \bigr) \times M_0$, and also that $M_k$ is an open subgroupoid of $M_{k + 1}$.  Thus it suffices to prove that $M_0$ is amenable.  We have
\[
M_0 = \bigl\{ [b^p,b^q,x] : x \in \Sigma,\ p,\; q \ge 0 \bigr\}.
\]
We claim that $[b^p,b^q,x] \mapsto p - q$ is a continuous homomorphism from $M_0$ to $\IZ$.  Since the kernel is just the diagonal $\{(x,x) : x \in \Sigma\}$, hence amenable, Proposition \ref{p.nuclearity} will imply that $M_0$ is amenable.

We have only to show that the homomorphism is well-defined.  Let $[b^p,b^q,C_m(i)] = [b^r,b^s,C_n(j)]$.  Then comparing sources gives $b^q C_m(i) = b^s C_n(j)$.  Without loss of generality we suppose that $q \le s$.  Then $C_m(i) = b^{s - q} C_n(j)$.  Comparing ranges, we then have $b^r C_n(j) = b^p C_m(i) = b^{p - q + s} C_n(j)$.  Since the action of $b$ on $\Lambda^*$ is injective, we have that $r = p - q + s$, hence $r - s = p - q$.
\end{proof}

We now turn to case (BS2).

\begin{Lemma}
\label{l.cltdsigma}
Suppose that we are in case (BS2).  Let $i \in [0,d)^\IN$.  For $n \in \IN$, let $C_n(i) = \bigcup_{k = 0}^\infty \bigl[ \alpha_{k - 1}(i) b^{nc} \bigr]$.  Put $s = \limsup_{\mu \to \infty} i_\mu$.
\begin{enumerate}
\item $C_n(i)$ is directed and hereditary, and $C_n(i) \subseteq C_{n+1}(i) \subseteq C_\infty(i)$ for $n \in \IN$.
\label{l.cltdsigmaone}
\item If $s \ge c$ then $C_0(i) = C_\infty(i)$.
\label{l.cltdsigmatwo}
\item If $s < c$ then $C_0(i) \not= C_1(i)$.  If $C$ is a directed hereditary set with $C_0(i) \subsetneq C$ then $C_1(i) \subseteq C$.
\label{l.cltdsigmathree}
\item Suppose $s < c$, and let $m = \bigl\lceil \tfrac{c - s}{d - c} \bigr\rceil$.  Then $C_0(i) \subsetneq C_1(i) \subsetneq \cdots \subsetneq C_{m - 1}(i) \subsetneq C_\infty(i)$, and these are the only directed hereditary sets of infinite height contained in $C_\infty(i)$.
\label{l.cltdsigmafour}
\end{enumerate}
\end{Lemma}

\begin{proof}
\eqref{l.cltdsigmaone}:  We have
\[
\alpha_k(i) b^{nc}
= \alpha_{k - 1}(i) b^{i_k} a b^{nc}
= \alpha_{k - 1}(i) b^{i_k + nd} a
= \alpha_{k - 1}(i) b^{nc} \cdot b^{i_k + n(d - c)} a.
\]
Therefore $[\alpha_{k - 1}(i) b^{nc}] \subseteq [\alpha_k(i) b^{nc}]$, and therefore the union defining $C_n(i)$ is increasing.  Hence $C_n(i)$ is a hereditary directed set.  Since $[\alpha_k(i) b^{nc}] \subseteq [\alpha_k(i) b^{(n + 1)c}]$, it follows that $C_n(i) \subseteq C_{n + 1}(i)$.  It is clear that $C_n(i) \subseteq C_\infty(i)$.

\eqref{l.cltdsigmatwo}:  We note that for any $\beta \in \Lambda$ we have $b^d \in [\beta a b^c]$ (e.g. from Lemma \ref{l.cltd}\eqref{l.cltd.a}, since $r(i)_k \le 1$ in that lemma).  Thus if $1 \le l_1 < \cdots < l_m$ are such that $i_{\ell_\mu} \ge c$ for $1 \le \mu \le m$, then $b^{md} \in [\alpha_{\ell_m}(i)]$.  If $p \ge 0$ is given, there is $m$ such that $md \ge p$.  Since $s \ge c$ there are $\ell_\mu$ as above, $1 \le \mu \le m$.  Then 
\[
b^p \in [b^{md}] \subseteq [\alpha_{\ell_m}(i)] \subseteq C_0(i).
\]
Therefore $B \subseteq C_0(i)$, and by Lemma \ref{l.containsb} it follows that $C_0(i) = C_\infty(i)$.

\eqref{l.cltdsigmathree}:  Since $s < c$ there is $\ell$ such that $i_\mu < c$ for $\mu \ge \ell$.  Let $\beta = \alpha_{\ell - 1}(i)$.  Since $b^{i_\ell} a \cdots b^{i_k} a$ is in form (R) for $k \ge \ell$, we have
\[
b^d \not\in \bigcup_{k \ge \ell} [b^{i_\ell} a \cdots b^{i_k} a]
= \sigma^\beta \bigcup_{k \ge \ell} [\alpha_k(i)]
= \sigma^\beta (C_0(i)).
\]
On the other hand,
\[
\beta b^{d + i_\ell} a
= \beta b^{i_\ell} a b^c
= \alpha_\ell(i) b^c
\in C_1(i),
\]
so that $b^d \in \sigma^\beta (C_1(i))$.  Therefore $C_0(i) \not= C_1(i)$.

Let $C\supsetneq C_0(i)$ be a directed hereditary set.  Choose $\gamma \in C \setminus C_0(i)$.  We may assume that $k = \theta(\gamma) \ge \ell$.  Then $\gamma = b^{i_0} a \cdots b^{i_k} a b^q$, where $q > i_{k + 1}$.  Let $i' = \sigma^{k + 1}(i)$ and $C' = \sigma^{\alpha_k(i)}(C)$.  Then $C_0(i') \subsetneq C'$ and $b^q \in C'$, where $q > i'_0 = i_{k+1}$.  By Lemma \ref{l.clessthand}\eqref{l.clessthand.c}, for each $\mu$,
\[
\alpha_\mu(i') b^c \in [b^q \vee \alpha_\mu(i')] \subseteq C'.
\]
Therefore $C_1(i') \subseteq C'$.  Hence
\[
C_1(i)
= \bigcup_\mu [\alpha_\mu(i) b^c]
= \alpha_k(i) \bigcup_\mu [\alpha_\mu(i') b^c]
\subseteq \alpha_k(i) C'
= C.
\]

\eqref{l.cltdsigmafour}:  We have that $s + (m - 1)(d - c) < c \le s + m(d - c)$.  Let $h$ be such that $i_k \le s$ for $k \ge h$.  Set $j = \sigma^h (i) \in [0,s]^\IN$.  Then for each $n$, $C_n(i) = \alpha_{h - 1}(i) C_n(j)$.   The definition of $\phi$ in the proof of Lemma \ref{l.cltd} implies that $\phi^n(j)_k = j_k + n(d - c)$, $k \ge 1$, for $0 \le n \le m$, and that $\phi^n(j) \in [0,c)^\IN$ for $n < m$, while $\phi^m(j)_k \ge c$ for infinitely many $k$.  Hence by Lemma \ref{l.cltd}\eqref{l.cltd.a}, for $0 \le n < m$ we have
\[
\alpha_k(j) b^{nc} = b^{nd} \alpha_k(\phi^n(j)), \text{ for } k \ge 1,
\]
and hence by \eqref{l.cltdsigmathree} that
\[
\sigma^{b^{nd}} C_n(j) 
= C_0(\phi^n(j)) 
\subsetneq C_1(\phi^n(j)) 
= \sigma^{b^{nd}} C_{n + 1}(j),
\]
with no directed hereditary set strictly between them.  Therefore $C_n(j) \subsetneq C_{n + 1}(j)$ with no directed hereditary set strictly between them.  Since $\phi^m(j)_k \ge c$ infinitely often, it follows from \eqref{l.cltdsigmatwo} that $C_m(j) = C_\infty(j)$, finishing the proof.
\end{proof}

\begin{Theorem}
\label{t.cltdsigmaamenable}
Suppose that we are in case (BS2).  Then $G |_\Sigma$ is amenable.  (In fact, it is Morita equivalent to the standard groupoid of $\CO_c$.)
\end{Theorem}

\begin{proof}
We let $Z = \bigl\{ C_0(i) : i \in [0,c)^\IN \bigr\} \subseteq \Sigma$.  Note that for $C \in \Sigma$, $C \in Z$ if and only if $b^d \not\in C$.  Thus $Z = (\widehat{b^d\Lambda})^c$ is a compact-open subset of $\Sigma$.  We claim that it is a transversal, in the sense of \cite{muhrenwil}.  To see this, let $C \in \Sigma$.  Then $m = \sup \{ \mu : b^\mu \in C \} < \infty$.  Let $C' = \sigma^{b^m} (C)$.  Then $b^d \not\in C'$, so $C' \in Z$.  Then $[e,b^m,C']$ has source $C$ and range in $Z$.  It follows from \cite{muhrenwil} that $G|_\Sigma$ is equivalent to $G|_Z$.

To analyze $G|_Z$ we first consider a pair $\beta \in \Lambda$ and $C \in \Sigma$ such that $\beta C \in Z$.  Since $b^d \not\in \beta C$, we must have $b^d \not\in C$, hence $C \in Z$.  Moreover, writing $\beta = b^{m_0} a \cdots b^{m_k}a b^{m_{k + 1}}$ in form (L), and $C = C_0(i)$ with $i \in [0,c)^\IN$, we must have $m_0$, $\ldots$, $m_k \in [0,c)$ and $m_{k + 1} \in [0,c - i_0)$.  Then $\beta C = C_0(f(m,i))$, where $f(m,i) = (m_0, \ldots, m_k, m_{k + 1} + i_0, i_1, i_2, \ldots)$.  Now, if $[\beta,\gamma, C] \in G|_Z$, let $\beta$ and $C$ be as above, and let $\gamma = b^{n_0} a \cdots b^{n_\ell} a b^{n_{\ell + 1}}$ in form (L).  Then the map $[\beta,\gamma,C] \mapsto \bigl(f(m,i), k - \ell, f(n,i)\bigr)$ defines an isomorphism of $G|_Z$ onto $\bigl\{ (x,p,y) \in [0,c)^\IN \times \IZ \times [0,c)^\IN : x_{\mu - p} = y_\mu$ for all large enough $\mu \bigr\}$, namely the standard groupoid for the Cuntz algebra $\CO_c$.  This groupoid is well-known to be amenable (e.g. \cite{ren}).  Therefore $G|_Z$ is amenable.
\end{proof}

\begin{Theorem}
\label{t.gisamenable}
$G$ is amenable.
\end{Theorem}

\begin{proof}
Since $G|_{\partial \Lambda}$ is amenable (Theorem \ref{t.boundaryamenable}), by \cite{ren} it suffices to prove that $G|_{(\partial \Lambda)^c}$ is amenable.  It is easy to see that $G|_\Lambda$ is amenable --- for any finitely aligned category of paths, $G|_\Lambda$ is the direct sum over $v \in \Lambda^0$ of the elementary groupoids $s^{-1}(v) \times s^{-1}(v)$.  Since $\Lambda / B$ and $\Sigma$ are disjoint relatively open subsets of $\Lambda^c$, it remains to show that $G|_{\Lambda / B}$ is amenable ($G_\Sigma$ was shown to be amenable in Theorems \ref{t.sigmaamenable} and \ref{t.cltdsigmaamenable}).  This follows since $G|_{\Lambda / B}$ is a transitive groupoid, with isotropy isomorphic to $\IZ$ (generated by $b$).
\end{proof}

We now apply Theorem \ref{t.restrictionboundary} to give  generators and relations for $C^*(\Lambda)$.

\begin{Theorem}
\label{t.baumsolrelations}
Let $\Lambda$ be the category of paths associated to the Baumslag-Solitar group $G$.  The representations of $C^*(\Lambda)$ are in one-to-one correspondence with pairs $\{ S_a, S_b \}$ of Hilbert space operators satisfying the relations
\begin{enumerate}
\item   \label{t.baumsolrelsa} 
$S_a$ and $S_b$ are isometries.
\item   \label{t.baumsolrelsb} 
$S_b$ is a unitary.
\item   \label{t.baumsolrelsc} 
$S_a S_b^c = S_b^d S_a$ in cases (BS1) and (BS2); $S_b^d S_a S_b^c = S_a$ in case (BS3).
\item   \label{t.baumsolrelsd} 
$\displaystyle \sum_{i = 0}^{d - 1} S_b^i S_a S_a^* S_b^{-i} = 1$.
\end{enumerate}
Moreover, in case (BS3), relation \eqref{t.baumsolrelsb} is redundant.
\end{Theorem}

\begin{proof}
By Theorem \ref{t.restrictionboundary}, we know that representations of $C^*(\Lambda)$ are in one-to-one correspondence with families $\{ S_\alpha : \alpha \in \Lambda \}$ satisfying the relations (1) - (4) of  Theorem \ref{t.restrictionboundary}.  Given such a family, we check that the operators $S_a$ and $S_b$ satisfy conditions \eqref{t.baumsolrelsa} - \eqref{t.baumsolrelsd}.  From Theorem \ref{t.restrictionboundary}\eqref{t.restrictionboundary.a} it follows that \eqref{t.baumsolrelsa} holds, and that $S_b$ is an isometry.  Then \eqref{t.baumsolrelsb} follows from Theorem \ref{t.restrictionboundary}\eqref{t.restrictionboundary.d} applied to the finite exhaustive set $\{ b \}$.  From Theorem \ref{t.restrictionboundary}\eqref{t.restrictionboundary.b} we conclude that \eqref{t.baumsolrelsc} holds.  Since $b^i a \perp b^j a$ when $0 \le i \not= j < d$, it follows from Theorem \ref{t.restrictionboundary}\eqref{t.restrictionboundary.c} that $S_b^i S_a$ and $S_b^j S_a$ have orthogonal ranges when $0 \le i \not= j < d$.  Theorem \ref{t.restrictionboundary}\eqref{t.restrictionboundary.d} applied to the finite exhaustive set $\{ b^i a : 0 \le i < d \}$ verifies \eqref{t.baumsolrelsd}.

Conversely, suppose that $S_a$ and $S_b$ are given satisfying \eqref{t.baumsolrelsa} - \eqref{t.baumsolrelsd}.  We define $S_\alpha$ for all $\alpha \in \Lambda$ by setting $S_e = 1$, and for $\alpha =  b^{i_0} a \cdots b^{i_k} a b^p$ in form (L), setting $S_\alpha = S_b^{i_0} S_a \cdots S_b^{i_k} S_a S_b^p$.  The proof (e.g. in \cite{scottwall}) of Proposition \ref{p.baumslagsolitar} uses only the relation $a b^c = b^{\pm d} a$.  Thus relation \eqref{t.baumsolrelsc} (and relation \eqref{t.baumsolrelsb} in case (BS3)) implies that Theorem \ref{t.restrictionboundary}\eqref{t.restrictionboundary.b} holds.  Relations \eqref{t.baumsolrelsa} and \eqref{t.baumsolrelsb} imply that Theorem \ref{t.restrictionboundary}\eqref{t.restrictionboundary.a} holds.  

Now we verify Theorem \ref{t.restrictionboundary}\eqref{t.restrictionboundary.c}.  Let $\alpha$, $\beta \in \Lambda$.  If $\alpha \perp \beta$, let them be written as in the statement of Proposition \ref{p.qlo}.  Then by that Proposition, there is $\ell \le \min \{s,t\}$ such that $e_\mu = f_\mu$ for $\mu < \ell$, and $e_\ell \not= f_\ell$.  Then we may compute:  $S_\alpha^* S_\beta = \cdots (b^{e_\ell} a)^* (b^{f_\ell} a) \cdots = 0$, since \eqref{t.baumsolrelsd} implies that $S_b^{e_\ell} S_a$ and $S_b^{f_\ell} S_a$ have orthogonal ranges.  Since $\alpha \perp \beta$, this verifies Theorem \ref{t.restrictionboundary}\eqref{t.restrictionboundary.c} in this case.  Suppose that $\alpha \Cap \beta$.  If e.g. $\beta \in \alpha\Lambda$, then we find that $S_\alpha S_\alpha^* S_\beta S_\beta^* = S_\beta S_\beta^*$, and $\beta = \alpha \vee \beta$.  Suppose instead that neither of $\alpha$ and $\beta$ extends the other.  The proof of Proposition \ref{p.qlo} shows that (without loss of generality) we may assume that $\alpha = \gamma b^m$ and $\beta = \gamma a b^{i_1} a \cdots b^{i_k} a b^q$, and that $\alpha \vee \beta = \gamma a b^{i_1} a \cdots b^{i_k} a b^p$ for some $p$.  Since $S_b S_b^* = 1$ by \eqref{t.baumsolrelsb}, the final factors of $S_b$ do not affect the computation of $SS^*$.  Hence $S_\alpha S_\alpha^* S_\beta S_\beta^* = S_\beta S_\beta^* = S_{\alpha \vee \beta} S_{\alpha \vee \beta}^*$.  

Finally we verify Theorem \ref{t.restrictionboundary}\eqref{t.restrictionboundary.d}.  Let $F$ be a finite exhaustive set.  If $v \in F$ then Theorem \ref{t.restrictionboundary}\eqref{t.restrictionboundary.d} is immediately satisfied.  Suppose $v \not\in F$.  First we suppose that $b^m \in F$ for some $m \ge 1$.  Then the right-hand side of Theorem \ref{t.restrictionboundary}\eqref{t.restrictionboundary.d} dominates $S_{b^m} S_{b^m}^* = 1$, by \eqref{t.baumsolrelsb}.  Now suppose that $F \cap B = \emptyset$.  Since $S_{\alpha b} S_{\alpha b}^* = S_\alpha S_\alpha^*$, we may assume that elements of $F$ have the form $b^{i_0} a \cdots b^{i_k} a$ (in form (L)).  Let us identify such elements with cylinder sets in $[0,d)^\IN$ via the sequences $(i_0, \ldots, i_k)$.  Moreover, because $F$ is exhaustive we have that these cylinder sets form a cover of $[0,d)^\IN$.   Thus we see that $\bigvee_{\alpha \in F} S_\alpha S_\alpha^* = 1$.

For the final statement of the theorem, assume that we are in case (BS3).  Since $S_a$ and $S_b$ are isometries, it follows from \eqref{t.baumsolrelsc} that  $S_b^c = S_a^* (S_b^*)^d S_a$.  Now we have from \eqref{t.baumsolrelsc}:
\[
S_a = S_b^d S_a S_b^c = S_b^d S_a S_a^* (S_b^*)^d S_a,
\]
and hence $S_a S_a^* = S_b^d S_a S_a^* (S_b^*)^d S_a S_a^*$.  Thus $S_a S_a^* \le S_b^d S_a S_a^* (S_b^*)^d$.  Now using \eqref{t.baumsolrelsd} gives
\[
S_b S_b^* = \sum_{i = 1}^d S_b^i S_a S_a^* (S_b^*)^i
= \sum_{i = 1}^{d - 1} S_b^i S_a S_a^* (S_b^*)^i + S_b^d S_a S_a^* (S_b^*)^d
\ge \sum_{i = 1}^{d - 1} S_b^i S_a S_a^* (S_b^*)^i + S_a S_a^* = 1.
\]
This proves \eqref{t.baumsolrelsb}.
\end{proof}

\begin{Remark}
The ``Cuntz-Krieger'' relation (Theorem \ref{t.restrictionboundary}\eqref{t.restrictionboundary.d}), defining the $C^*$-algebra from the Toeplitz $C^*$-algebra of $\Lambda$, is represented by \eqref{t.baumsolrelsb} and the equality (as opposed to $\le$) in \eqref{t.baumsolrelsd}.  We point out here that  if the group falls under case (BS1) or (BS2), then both of these relations are necessary.  To see this, we consider the representations of $\CT C^*(\Lambda)$ on $\ell^2(\Lambda / B)$ and on $\ell^2(\Sigma)$.  The first of these satisfies \eqref{t.baumsolrelsb} but not \eqref{t.baumsolrelsd}; the second satisfies \eqref{t.baumsolrelsd} but not \eqref{t.baumsolrelsb}.
\end{Remark}

\begin{Remark}
\label{r.katsuraexamples}
The relations \ref{t.restrictionboundary}\eqref{t.restrictionboundary.a} - \eqref{t.restrictionboundary.d} are the same as those found by Katsura (\cite{kat}, Example A.6).  Specifically, our $C^*(\Lambda)$ is isomorphic to Katsura's $\CO(E_{n,m})$ for $m \not= 0$, with the identifications $d = n$, and $c = m$ in case (BS1) or (BS2) when $m > 0$, and $c = -m$ in case (BS3) when $m < 0$.   We remark that in the case where $m < 0$, our analysis shows that one of the relations for the $C^*$-algebra turns out to be redundant.  (In the case that $m = 0$, the group is not one of those discussed in \cite{baumsol}.  In fact, if $c = 0$, the group becomes $\IZ * \IZ / d\IZ$.  If $d > 1$ then $\Lambda$ is not a category of paths, since \eqref{d.categoryofpathsc} fails.  If $d = 1$, then $G = \IZ$, and $\Lambda$ is the path category of the directed graph having one vertex and one edge, giving the same result as \cite{kat}.)
\end{Remark}

\section{$K$-theory}
\label{s.ktheory}

Our next task is to compute the $K$-theory of $C^*(\Lambda) = C^*(G|_{\partial\Lambda})$.  We give a different calculation than that of \cite{kat}.  Thus we also compute the $K$-theory of the core algebra.  Let $A = C^*(\Lambda)$.  We let $\gamma$ denote the gauge action of $\IT$ on $A$ induced by the cocycle $\theta : \Lambda \to \IZ$.  Then Takai-Takesaki duality implies that $A \otimes \CK \cong (A \times_\gamma \IT) \times_{\widehat{\gamma}} \IZ$.

\begin{Lemma}
\label{l.fixedpointalgebra}
The fixed-point algebra $A^\gamma$ is Morita equivalent to $A \times_\gamma \IT$.
\end{Lemma}

\begin{proof}
Let $\zeta \in C(\IT)$ be the function $\zeta(z) = z$.  Then the collection $\{ \zeta^n S_\alpha S_\beta^* : n \in \IZ,\ \alpha,\ \beta \in \Lambda \} \subseteq C(\IT,A) \subseteq A \times_\gamma \IT$ is a total set.  A short calculation in the convolution algebra $C(\IT,A)$ shows that $(\zeta^m S_\alpha S_\beta^*)(\zeta^n S_\mu S_\nu^*) = \delta_{m,n - \theta(\mu) + \theta(\nu)} \zeta_n S_\alpha S_\beta^* S_\mu S_\nu^*$.  Of course, the collection $\{ S_\alpha S_\beta^* : \theta(\alpha) = \theta(\beta) \}$ is a total set in $A^\gamma$.  Now let $n \in \IZ$ and $\alpha$, $\beta \in \Lambda$.  Choose $k \ge 0$ such that $n + \theta(\beta) + k \ge 0$, and let $\nu \in \Lambda$ with $\theta(\nu) = n + \theta(\beta) + k$.  Let $M = \{ b^{i_1} a \cdots b^{i_k} a : i_j \in [0,d) \text{ for } 1 \le j \le k \}$.  Then $S_\alpha S_\beta^* = \sum_{\mu \in M} S_{\alpha \mu} S_{\beta \mu}^*$.  We have that
\[
\zeta^n S_\alpha S_\beta^*
= \sum_{\mu \in M} \zeta^n S_{\alpha \mu} S_{\beta \mu}^* 
= \sum_{\mu \in M} (S_{\alpha \mu} S_\nu^*) (S_\nu S_\nu^*) (\zeta^n S_\nu S_{\beta \mu}^*)
\]
is in the ideal generated by $A^\gamma$.  Thus $A^\gamma$ is a full hereditary subalgebra of $A \times_\gamma \IT$.
\end{proof}

We next compute the $K$-theory of $A^\gamma$.  We know that $A^\gamma = C^*(H)$, where $H = G^\gamma |_{\partial \Lambda}$.  Recall from the proof of Theorem \ref{t.boundaryamenable} that $H = \bigcup_{n = 0}^\infty H_n$, that $H_0 \cong \IZ \ltimes \partial \Lambda$, and that $H_n \cong ([0,d)^n \times [0,d)^n) \times H_0$.  So we begin with the computation of $K_*(C^*(H_0))$.  Since $C^*(H_0) = C(\partial \Lambda) \times_b \IZ$, and $\partial \Lambda$ is totally disconnected, we obtain from the Pimsner-Voiculescu exact sequence:
\[
0 \longrightarrow  K_1(C^*(H_0))  \longrightarrow  K_0(C(\partial \Lambda))  \xrightarrow{id - b_*}  K_0(C(\partial \Lambda)) \longrightarrow  K_0(C^*(H_0)) \longrightarrow 0
\]
\vspace*{.05 in}

Since $K_0(C(\partial \Lambda)) \cong C(\partial \Lambda, \IZ)$, it follows that $K_1(C^*(H_0)) \cong \IZ \text{-span} \{ \chi_E : E \subseteq \partial \Lambda$ is compact-open and $b$-invariant$ \}$, and $K_0(C^*(H_0)) \cong C(\partial \Lambda, \IZ) \big/ \IZ \text{-span} \{\chi_E - \chi_{bE} : E \subseteq \partial \Lambda \text{ compact-open} \}$.  We recall the action of $b$ on $\partial \Lambda = [0,d)^\IN$ from Lemmas \ref{l.cged}, \ref{l.cltd} and \ref{l.casethree}:  add 1 in the 0th coordinate; $d$ in the $j$th coordinate carries as $\pm c$ to the $(j+1)$st coordinate, using $+c$ in cases (BS1) and (BS2), and $-c$ in case (BS3) ($b^d$ acts as $\phi^{-1}$).  

\begin{Lemma}
\label{l.openinvariant}
Let $e = (c,d)$ be the greatest common divisor of $c$ and $d$ and let $c = c' e$, $d = d' e$.  For $\mu_1$, $\ldots$, $\mu_k \in [0,e)$, let $U(\mu_1, \ldots, \mu_k) = \{ i \in [0,d)^\IN : i_\ell \equiv \mu_\ell \pmod{e}, \ 1 \le \ell \le k \}$.  Then $U(\mu_1, \ldots, \mu_k)$ is $b$-invariant, and every open $b$-invariant set is a union of such sets.
\end{Lemma}
\begin{proof}
Since the complement of $U(\mu_1, \ldots, \mu_k)$ is a union of such sets, it is enough to show that $b \cdot U(\mu_1, \ldots, \mu_k)$\allowbreak$ \subseteq U(\mu_1, \ldots, \mu_k)$; moreover, since the 0th coordinate is unrestricted, it is enough to check that $b^d \cdot U(\mu_1, \ldots, \mu_k)$\allowbreak$ \subseteq U(\mu_1, \ldots, \mu_k)$.  Let $i \in U(\mu_1, \ldots, \mu_k)$, and let $i' = b^d \cdot i$.  Then for each $\ell$ there is $k_\ell$ such that  $i'_\ell \equiv i_\ell + k_\ell c \pmod{d}$.  Since $e$ divides both $c$ and $d$, it follows that $i'_\ell \equiv i_\ell \equiv \mu_\ell \pmod{e}$ for $1 \le \ell \le k$.  Thus $i' \in U(\mu_1, \ldots, \mu_k)$.  Therefore $U(\mu_1, \ldots, \mu_k)$ is $b$-invariant.

Let us write $Z(j_0, \ldots, j_k)$ for the cylinder set $\{ i \in \partial \Lambda : i_\ell = j_\ell \text{ for } 0 \le \ell \le k \}$ in $\partial \Lambda$.  We claim that for $\mu_1$, $\ldots$, $\mu_k \in [0,e)$, and for any $j_0$, $U(\mu_1, \ldots, \mu_k) = \bigcup_{n \in \IZ} b^n \cdot Z(j_0, \mu_1, \ldots, \mu_k)$.  The containment $\supseteq$ follows from the fact that $U(\mu_1, \dots, \mu_k)$ is invariant and contains $Z(j_0, \mu_1, \ldots, \mu_k)$.  To see the containment $\subseteq$, notice first that we may adjust the 0th coordinate arbitrarily.   Viewing the remaining coordinates as copies of $\IZ / d \IZ$, note that addition by $c$ has $d'$ orbits.  Thus we may adjust the first coordinate to any element congruent to $\mu_1$ modulo $e$.  Then adding a multiple of $cd'$ will not further change the first coordinate.  Since $cd' = c'd$, this amounts to adding $cc'$ in the second coordinate.  Again, since $(cc',d) = e$, we may adjust the second coordinate to any element congruent to $\mu_2$ modulo $e$ without changing the first coordinate.  Repeating this argument, we see that  we may fill up $U(\mu_1, \ldots, \mu_k)$ by applying $b$ repeatedly to $Z(j_0, \mu_1, \ldots, \mu_k)$, proving the claim.

Now let $V$ be an open $b$-invariant set.  We write $V = \bigcup_p Z(p)$ as a union of cylinder sets, where the $p$ are tuples from $[0,d)$.  Then
\[
V
= \bigcup_{n \in \IZ} b^n \cdot V
= \bigcup_{n \in \IZ} \bigcup_p b^n \cdot Z(p)
= \bigcup_p U(p'),
\]
where if $p = (j_0, j_1, \ldots, j_k)$ then  $p' = (\mu_1, \ldots, \mu_k)$ for $\mu_\ell \in [0,e)$, $\mu_\ell \equiv j_\ell \pmod{\ell}$, $1 \le \ell \le k$.
\end{proof}

\begin{Definition}
\label{d.notation}
We will use the following notation.  For $j \in [0,d)^{k+1}$ let $Q_{Z(j_0, \ldots, j_k)} = S_{b^{j_0} a \cdots b^{j_k} a} S_{b^{j_0} a \cdots b^{j_k} a}^*$ in $C^*(H_0)$.  If $Z \subseteq \partial \Lambda$ is a compact-open subset, then $Z$ is a finite disjoint union of cylinder sets.  Since $Z(i_0, \ldots, i_k) = \bigsqcup_{j = 0}^{d - 1} Z(i_0, \ldots, i_k, j)$, and $Q_{Z(i_0, \ldots, i_k)} = \sum_{j = 0}^{d - 1} Q_{Z(i_0, \ldots, i_k, j)}$, we may define $Q_Z = \sum_j Q_{Z_j}$ if $Z = \bigsqcup_j Z_j$ for any finite disjoint collection of cylinder sets $\{ Z_j \}$.
\end{Definition}

Note that if $Z \subseteq \partial \Lambda$ is an invariant compact-open set, then $Q_Z$ and $S_b$ commute.  In the next few items we write $\IZ_e$ for  $(\IZ / e \IZ)^{\IZ^+} \cong [0,e)^{\IZ^+}$, though it is only the coordinate-wise group structure that will be convenient (and we omit the 0th coordinate).

\begin{Corollary}
\label{c.koneofhzero}
$K_1(C^*(H_0)) \cong C(\IZ_e, \IZ)$.
\end{Corollary}

We note that under this isomorphism, we have that $[S_b Q_{U(\mu_1, \ldots, \mu_k)}]$ corresponds to $\chi_{Z(\mu_1 + e \IZ, \ldots, \mu_k + e \IZ)}$.

\begin{Lemma}
\label{l.kzeroofhzero}
$K_0(C^*(H_0)) \cong C(\IZ_e, \IZ[ \tfrac{1}{d'}])$.
\end{Lemma}

\begin{proof}
We define a homomorphism $C(\partial \Lambda, \IZ) \to C(\IZ_e, \IZ[ \tfrac{1}{d'}])$ by $\chi_{Z(j_0, \ldots, j_k)} \mapsto (d')^{-k} \chi_{Z(j_1 + e\IZ, \dots, j_k + e\IZ)}$.  To see that it is well-defined, we note that we have to check that the relation $Z(j_0, \ldots, j_k) = \bigcup_{\ell = 0}^{d - 1} Z(j_0, \ldots, j_k, \ell)$ is respected.  For this we compute
\begin{align*}
\sum_{\ell = 0}^{d - 1} \chi_{Z(j_0, \ldots, j_k, \ell)} 
&\mapsto \sum_{\ell = 0}^{d - 1} (d')^{-k - 1} \chi_{Z(j_1 + e\IZ, \ldots, j_k + e\IZ, \ell + e\IZ)} \\
&= \sum_{\ell' = 0}^{e - 1} (d')^{-k - 1} d' \chi_{Z(j_1 + e\IZ, \ldots, j_k + e\IZ, \ell' + e\IZ)} 
= (d')^{-k} \chi_{Z(j_1 + e\IZ, \ldots, j_k + e\IZ)}.
\end{align*}
The map is clearly surjective, and its kernel contains $\chi_Z - \chi_{b \cdot Z}$ for every cylinder set $Z$.  We claim that its kernel is generated by the functions of the form $\chi_Z - \chi_{b \cdot Z}$; this will conclude the proof.  Let $f$ be an element of the kernel.  We may choose $k$ such that $f$ is a linear combination of characteristic functions of cylinder sets of length $k$:  $f = \sum_j n_j \chi_{Z(j)}$, where $j$ ranges over $[0,d)^{k + 1}$.  Let $[0,d)^{k + 1} = \bigsqcup_p E_p$ be the equivalence classes defined by congruence modulo $e$ in coordinates 1 through $k$.  Then $f$ maps to $(d')^{-k} \sum_p \bigl( \sum_{j \in E_p} n_j \bigr) \chi_{Z(p)}$.  Since $f$ is in the kernel, we have that $\sum_{j \in E_p} n_j = 0$ for each $p$.  But then $f = \sum_p \bigl( \sum_{j \in E_p} n_j \chi_{Z(j)} \bigr)$, and it is easy to see that each inner sum is in the span of the functions of the form $\chi_Z - \chi_{b \cdot Z}$.
\end{proof}

The computation of the $K$-theory of $A^\gamma$ uses the following elementary lemma.

\begin{Lemma}
\label{l.sublemma}
Let $M$ and $N$ be abelian groups, and $\eta : M \to M$, $\xi : N \to N$, and $I : M \to N$ homomorphisms, such that $I \circ \eta = \xi \circ I$.  Let $\displaystyle \widetilde{M} = \lim_{\buildrel \longrightarrow \over \eta} M$, $\displaystyle \widetilde{N} = \lim_{\buildrel \longrightarrow \over \xi} N$, and $\displaystyle \widetilde{I} = \lim_{\longrightarrow} I : \widetilde{M} \to \widetilde{N}$.  Suppose that
\begin{enumerate}
\item \label{l.sublemmaone}  $\ker(I) = \bigcup_n \ker(\eta^n)$.
\item \label{l.sublemmatwo} $N = \bigcup_n \xi^{-n}(I(M))$.
\end{enumerate}
Then $\widetilde{I}$ is an isomorphism.
\end{Lemma}

\begin{proof}
We let $M^{(n)}$ denote the $n$th copy of $M$ in the inductive limit, etc.  Let $x \in \ker(\widetilde{I})$.  Choose $x_n \in M^{(n)}$ with $x_n \mapsto x$.  There is $k$ such that $\xi^k (I(x_n)) = 0$.  Let $x_{n+k} = \eta^k(x_n)$.  Then $x_{n+k} \in \ker(I)$, so by \eqref{l.sublemmaone} there is $\ell$ such that $x_{n+k} \in \ker \eta^\ell$.  Then $0 = \eta^\ell(x_{n+k}) \mapsto x$, so $x = 0$.

Let $y \in \widetilde{N}$.  Choose $y_n \in N^{(n)}$ with $y_n \mapsto y$.  By \eqref{l.sublemmatwo} there is $k$ such that $y_n \in \xi^{-k}(I(M))$.  Then there is $z_{n+k} \in M^{(n+k)}$ such that $\xi^k(y_n) = I(z_{n+k})$.  Let $z_{n+k} \mapsto z \in \widetilde{M}$.  Then $y_n \mapsto \widetilde{I}(z)$, so that $y = \widetilde{I} (z)$.
\end{proof}

\begin{Theorem}
\label{t.kthyfixedpoint}
The $K$-theory of $A^\gamma$ is given by
\[
K_0(A^\gamma) \cong \IZ[\tfrac{1}{d}] \qquad \text{ and } \qquad
K_1(A^\gamma) \cong \IZ[\tfrac{1}{c}].
\]
Moreover, the generator $d^{-k}$ is represented in $K_0$ by $[S_{b^{j_0} a \cdots b^{j_k} a} S_ {b^{j_0} a \cdots b^{j_k} a}^*]$, and the generator $c^{-k}$ is represented in $K_1$ by $[S_{b^{i_1} a \cdots b^{i_k} a} S_b S_{b^{i_1} a \cdots b^{i_k} a}^*]$.
\end{Theorem}

\begin{proof}
Recall from the proof of Theorem \ref{t.boundaryamenable} that $C^*(H_n) \cong M_{d^n} \otimes C^*(H_0)$.  Explicitly, we have that $C^*(H_0) \cong C(\partial \Lambda) \times_b \IZ$ via
\[
\chi_{[e,e,\alpha \partial \Lambda]} \longleftrightarrow S_\alpha S_\alpha^* \qquad \text{ and } \qquad
\chi_{[b,e,\partial \Lambda]} \longleftrightarrow  S_b
\]
(where we view $C(\partial \Lambda) \times_b \IZ \subseteq C^*(\Lambda)$ by means of the generators of Theorem \ref{t.baumsolrelations}.  The inclusion $C^*(H_0) \hookrightarrow C^*(H_1) \cong M_d \otimes C^*(H_0)$ is described on these generators as follows.  Let $\alpha = b^{i_0} a \cdots b^{i_m} a$, and $\alpha' = b^{i_1} a \cdots b^{i_m} a$.  Then
\begin{align*}
S_\alpha S_\alpha^*
&\longleftrightarrow \chi_{[e,e,\alpha \partial \Lambda]} 
= \chi_{[b^{i_0} a, b^{i_0} a, \alpha' \partial \Lambda]} 
\longleftrightarrow e_{i_0,i_0} \otimes S_{\alpha'} S_{\alpha'}^*, \\
S_b
&\longleftrightarrow \chi_{[b,e,\partial \Lambda]} 
= \sum_{\ell = 0}^{d - 1} \chi_{[b,e,b^\ell a \partial \Lambda]} 
= \sum_{\ell = 0}^{d - 1} \chi_{[b^{\ell + 1} a, b^\ell a, \partial \Lambda]} \\
&\qquad \qquad = \sum_{\ell = 0}^{d - 2} \chi_{[b^{\ell + 1} a, b^\ell a, \partial \Lambda]} + \chi_{[a b^{\pm c}, b^{d - 1} a, \partial \Lambda]} 
\longleftrightarrow \sum_{\ell = 0} ^{d - 2} e_{\ell + 1, \ell} \otimes 1 + e_{0, d - 1} \otimes S_b^{\pm c},
\end{align*}
where the exponent $+c$ is used in cases (BS1) and (BS2), while $-c$ is used in case (BS3).  This dichomoty will continue throughout the proof.  In general, the inclusion $M_{d^n} \otimes C^*(H_0) \cong C^*(H_n) \hookrightarrow C^*(H_{n + 1}) \cong M_{d^{n + 1}} \otimes C^*(H_0)$ is given by tensoring by $M_{d^n}$ on the left of the above inclusion.

For the computation of $K_0(A^\gamma)$, we consider the map on $C(\IZ_e,\IZ[\tfrac{1}{d'}])$ induced by the above inclusion. The generator $\chi_{Z(\mu_1 + e\IZ, \ldots, \mu_k + e\IZ)}$ corresponds to $(d')^k [S_{b^{j_0} a b^{\mu_1} a \cdots b^{\mu_k} a} S_{b^{j_0} a b^{\mu_1} a \cdots b^{\mu_k} a}^*]$, for any choice of $j_0$.  The above inclusion sends this to
\[
(d')^k [e_{j_0,j_0} \otimes S_{b^{\mu_1} a \cdots b^{\mu_k} a} S_{b^{\mu_1} a \cdots b^{\mu_k} a}^*]
= (d')^k [S_{b^{\mu_1} a \cdots b^{\mu_k} a} S_{b^{\mu_1} a \cdots b^{\mu_k} a}^*],
\]
which corresponds to $(d')^k (d')^{-k + 1} \chi_{Z(\mu_2 + e\IZ, \ldots, \mu_k + e\IZ)} = d' \chi_{Z(\mu_2 + e\IZ, \ldots, \mu_k + e\IZ)}$.  Thus we have
\[
K_0(A^\gamma) = \lim_{\buildrel \longrightarrow \over {\eta_0}} C(\IZ_e,\IZ[\tfrac{1}{d'}]),
\]
via the map $\eta_0 : \chi_{Z(\mu_1 + e\IZ, \ldots, \mu_k + e\IZ)} \mapsto d' \chi_{Z(\mu_2 + e\IZ, \ldots, \mu_k + e\IZ)}$.  We will find $K = \bigcup_n \ker(\eta_0^n)$.  Let $\IZ_e$ have the usual product measure, assigning measure $e^{-k}$ to the set $Z(\mu_1 + e\IZ, \ldots, \mu_k + e\IZ)$.  We claim that $K = \{ f \in C(\IZ_e, \IZ[\tfrac{1}{d'}]) : \int f = 0 \}$.  To see this, first let $\int f = 0$.  Choose $k$ so that $f$ is constant on each cylinder set of length $k$.  Thus $f = \sum_{\mu_1, \ldots, \mu_k \in [0,e)} c_{\mu_1, \ldots, \mu_k} \chi_{Z(\mu_1 + e\IZ, \ldots, \mu_k + e\IZ)}$.  Then
\[
0 = \int f = e^{-k} \sum_{\mu_1, \ldots, \mu_k \in [0,e)} c_{\mu_1, \ldots, \mu_k},
\]
and hence
\[
\eta_0^k(f) = e^{-k} (d')^k \sum_{\mu_1, \ldots, \mu_k \in [0,e)} c_{\mu_1, \ldots, \mu_k} = 0.
\]
Conversely, let $f \in \ker(\eta_0^n)$ for some $n$.  Choose $k \ge n$ such that $f$ is constant on cylinder sets of length $k$.  Thus 
\[
f = \sum_{\mu_1, \ldots, \mu_k \in [0,e)} c_{\mu_1, \ldots, \mu_k} \chi_{Z(\mu_1 + e\IZ, \ldots, \mu_k + e\IZ)},
\]
and hence
\begin{align*}
0 = \eta_0^n(f) 
&= (d')^n \sum_{\mu_1, \ldots, \mu_k \in [0,e)} c_{\mu_1, \ldots, \mu_k} \chi_{Z(\mu_{n+1} + e\IZ, \ldots, \mu_k + e\IZ)} \\
&= (d')^n \sum_{\mu_{n + 1}, \ldots, \mu_k \in [0,e)} \bigl( \sum_{\mu_1, \ldots, \mu_n \in [0,e)} c_{\mu_1, \ldots, \mu_k} \bigr) \chi_{Z(\mu_{n + 1} + e\IZ, \ldots, \mu_k + e\IZ)}. 
\end{align*}
Thus the inner sum vanishes for each choice of $\mu_{n + 1}$, $\ldots$, $\mu_k$.  Therefore
\[
\int f = e^{-k} \sum_{\mu_1, \ldots, \mu_k \in [0,e)} c_{\mu_1, \ldots, \mu_k} 
= e^{-k} \sum_{\mu_{n + 1}, \ldots, \mu_k \in [0,e)} \bigl( \sum_{\mu_1, \ldots, \mu_n \in [0,e)} c_{\mu_1, \ldots, \mu_k} \bigr) 
= 0.
\]

Thus Lemma \ref{l.sublemma}\eqref{l.sublemmaone} holds, where $M = C(\IZ_e,\IZ[\tfrac{1}{d'}])$, $N = \IZ[\tfrac{1}{d}]$, $I$ is integration, $\eta = \eta_0$, and $\xi = d\cdot$.  Since $\xi$ is surjective, Lemma \ref{l.sublemma}\eqref{l.sublemmatwo} holds.  Therefore Lemma \ref{l.sublemma} implies that $K_0(A^\gamma) \cong \IZ[\tfrac{1}{d}]$.  From Lemma \ref{l.kzeroofhzero} we have that $[S_{b^{j_0} a \cdots b^{j_k} a} S_ {b^{j_0} a \cdots b^{j_k} a}^*] = (d')^{-k} \chi_{Z(j_1 + e\IZ, \ldots, j_k + e\IZ)}$ in $K_0(C^*(H_0))$.  The integral of this function is $(d')^{-k} e^{-k} = d^{-k}$, thus identifying generators of $K_0(A^\gamma)$.

For the computation of $K_1(A^\gamma)$ we consider the map on $C(\IZ_e, \IZ)$ induced by the inclusion of $C^*(H_n)$ into $C^*(H_{n + 1})$.  Recalling Definition \ref{d.notation}, the inclusion gives

\begin{align*}
Q_{U(\mu_1, \ldots, \mu_k)}
&= \sum_{i_0 = 0}^{d - 1} \sum_{j_1, \ldots, j_k = 0}^{d' - 1} S_{b^{i_0} a b^{\mu_1 + c j_1} a \cdots b^{\mu_k + c j_k} a} S_{b^{i_0} a b^{\mu_1 + c j_1} a \cdots b^{\mu_k + c j_k} a}^* \\
&\mapsto \sum_{i_0 = 0}^{d - 1} e_{i_0, i_0} \otimes \sum_{j_1, \ldots, j_k = 0}^{d' - 1} S_{b^{\mu_1 + c j_1} a \cdots b^{\mu_k + c j_k} a} S_{b^{\mu_1 + c j_1} a \cdots b^{\mu_k + c j_k} a}^* \\
&= 1 \otimes \sum_{j_1, \ldots, j_k = 0}^{d' - 1} S_{b^{\mu_1 + c j_1} a \cdots b^{\mu_k + c j_k} a} S_{b^{\mu_1 + c j_1} a \cdots b^{\mu_k + c j_k} a}^* 
= 1 \otimes Q_{\widetilde{U}(\mu_1, \ldots, \mu_k)},
\end{align*}

where $\widetilde{U}(\mu_1, \ldots, \mu_k) = \bigcup_{j_1, \ldots, j_k = 0}^{d' - 1} Z(\mu_1 + c j_1, \ldots, \mu_k + c j_k)$.  Thus the inclusion gives (in $d \times d$ matrices)
\[
S_b \chi_{U(\mu_1, \ldots, \mu_k)} \longmapsto
\begin{pmatrix}
0 & & & S_b^{\pm c} \\
1 & 0 & & \\
&\ddots & \ddots & \\
& & 1 & 0
\end{pmatrix}
\begin{pmatrix}
Q_{\widetilde{U}(\mu_1, \ldots, \mu_k)} & & &\\
& \ddots & & \\
& & \ddots & \\
& & & Q_{\widetilde{U}(\mu_1, \ldots, \mu_k)}
\end{pmatrix}.
\]
Hence 
\[
(S_b Q_{U(\mu_1, \ldots, \mu_k)})^d 
\longmapsto 1 \otimes S_b^{\pm c} Q_{\widetilde{U}(\mu_1, \ldots, \mu_k)}
= 1 \otimes (S_b^e Q_{\widetilde{U}(\mu_1, \ldots, \mu_k)})^{\pm c'},
\]
and therefore $[S_b Q_{U(\mu_1, \ldots, \mu_k)}] \longmapsto \pm c' [S_b^e Q_{\widetilde{U}(\mu_1, \ldots, \mu_k)}]$.  Let $E_i = b^i \widetilde{U}(\mu_1, \ldots, \mu_k)$ for $0 \le i < e$.  Then the $E_i$ are pairwise disjoint, and $\bigcup_{i = 0}^{e - 1} E_i = U(\mu_2, \ldots, \mu_k)$.  We claim that $[S_b^e Q_{\widetilde{U}(\mu_1, \ldots, \mu_k)}] = [S_b Q_{U(\mu_2, \ldots, \mu_k)}]$.  To see this, we first define a *-homomorphism $\tau : M_e(\IC) \to A^\gamma$ by $\tau(e_{ij}) = S_b^i Q_{E_0} S_b^{-j}$.  Let $v = \sum_{i = 0}^{e - 2} e_{i + 1,i} + e_{0, e - 1}$ be the shift matrix.  Then $\tau(v) = \sum_{i = 0}^{e - 2} S_b^{i + 1} Q_{E_0} S_b^{-j} + Q_{E_0} S_b^{-e + 1}$.   Let $v_t$, $0 \le t \le 1$ be a continuous path of unitary matrices from 1 to $v$.  We have
\[
S_b Q_{U(\mu_2, \ldots, \mu_k)} 
= \sum_{i = 0}^{e - 1} S_b Q_{E_i} 
= \sum_{i = 0}^{e - 1} S_b^{i + 1} Q_{E_0} S_b^{-i} 
= \sum_{i = 0}^{e - 2} \tau(e_{i+1,i}) + S_b^e \tau(e_{0,e-1}) 
= (S_b^e Q_{E_0} + \sum_{i = 1}^{e - 1} Q_{E_i}) \tau(v).
\]
Thus $S_b Q_{U(\mu_2, \ldots, \mu_k)}  \tau(v_t)^*$ is a continuous path from $S_b Q_{U(\mu_2, \ldots, \mu_k)} $ to $S_b^e Q_{E_0} + \sum_{i = 1}^{e - 1} Q_{E_i}$, proving the claim.  Now we find that
\begin{align*}
\chi_{Z(\mu_1 + e\IZ, \ldots, \mu_k + e\IZ)}
&\longleftrightarrow [S_b Q_{U(\mu_1, \ldots, \mu_k)}] \\
&\longmapsto \pm c'[S_b^e Q_{\widetilde{U}(\mu_1, \ldots, \mu_k)}] 
= \pm c'[S_b Q_{U(\mu_2, \ldots, \mu_k)}] 
\longleftrightarrow \pm c' \chi_{Z(\mu_2 + e\IZ, \ldots, \mu_k + e\IZ)}.
\end{align*}
Thus we have
\[
K_1(A^\gamma) = \lim_{\buildrel \longrightarrow \over {\eta_1}} C(\IZ_e,\IZ),
\]
via the map $\eta_1 : \chi_{Z(\mu_1 + e\IZ, \ldots, \mu_k + e\IZ)} \mapsto c' \chi_{Z(\mu_2 + e\IZ, \ldots, \mu_k + e\IZ)}$.  Note that
\[
\int \eta_1 (\chi_{Z(\mu_1 + e\IZ, \ldots, \mu_k + e\IZ)}) = \int \pm c' \chi_{Z(\mu_2 + e\IZ, \ldots, \mu_k + e\IZ)}
= \pm c' e^{-k + 1} = \pm c e^{-k} = \pm c \int \chi_{Z(\mu_1 + e\IZ, \ldots, \mu_k + e\IZ)}.
\]
Thus $I \circ \eta = \xi \circ I$, where $M = C(\IZ_e, \IZ)$, $N = \IZ[\tfrac{1}{c}]$, $\eta = \eta_1$, $\xi = \pm c \cdot$, and $I$ is given by integration.  Essentially the same computation as for $K_0$ shows that Lemma \ref{l.sublemma}\eqref{l.sublemmaone} holds.  Since Lemma \ref{l.sublemma}\eqref{l.sublemmatwo} clearly holds, we have that $K_1(A^\gamma) \cong \IZ[\tfrac{1}{c}]$.  Finally, we have $M^{(n)} = K_1(C^*(H_n)) \cong C(\IZ_e, \IZ)$, where 
\[
C(\IZ_e,\IZ) \ni 1 = [e_{i_1,i_1} \otimes \cdots \otimes e_{i_n,i_n} \otimes S_b] = [S_{b^{i_1} a \cdots b^{i_n} a} S_b S_{b^{i_1} a \cdots b^{i_n} a}^*] \longmapsto {\pm c}^{-n},
\]
thus identifying the generators of $K_1(A^\gamma)$.
\end{proof}

\begin{Theorem}
\label{t.kthy}
The $K$-theory of $C^*(\Lambda)$ is given as follows.
\begin{enumerate}
\item If either $c > 1$ or we are in case (BS3), and if $d > 1$, then $K_0(C^*(\Lambda)) \cong \IZ / (d - 1)\IZ$ and $K_1(C^*(\Lambda)) \cong \IZ / (\pm c - 1)\IZ$, where the minus sign is used in case (BS3).
\item If either $c > 1$ or we are in case (BS3), and if $d = 1$, then $K_0(C^*(\Lambda)) \cong \IZ$ and $K_1(C^*(\Lambda)) \cong \IZ / (\pm c - 1) \IZ \oplus \IZ$, where the minus sign is used in case (BS3).
\item If $c = 1$ in case (BS2) (so $d > 1$), then $K_0(C^*(\Lambda)) \cong \IZ / (d - 1)\IZ \oplus \IZ$ and $K_1(C^*(\Lambda)) \cong \IZ$.
\item If $c = 1$ in case (BS1) (so $d = 1$), then $K_0(C^*(\Lambda)) \cong \IZ \oplus \IZ$ and $K_1(C^*(\Lambda)) \cong \IZ \oplus \IZ$.
\end{enumerate}
In the first two cases, the class $[1]$ of the identity is given by 1, while in the last two cases it is given by $(1,0)$.
\end{Theorem}

\begin{proof}
Recall from Lemma \ref{l.fixedpointalgebra} that $A^\gamma$ is Morita equivalent to $A \times_\gamma \IT$.  In $A \times_\gamma \IT$ we have partial isometries $\zeta S_{b^i a}$, $0 \le i < d$, with $(\zeta S_{b^i a})^* (\zeta S_{b^i a}) = \zeta 1$ and $(\zeta S_{b^i a}) (\zeta S_{b^i a})^* = S_{b^i a} S_{b^i a}^*$.  Since $\widehat{\gamma} (\zeta^n S_\alpha S_\beta^*) = \zeta^{n + 1} S_\alpha S_\beta^*$, it follows that $\widehat{\gamma}_*$ is given by multiplication by $d^{-1}$ in $K_0$.  To calculate the effect of $\widehat{\gamma}_*$ on $K_1$, we consider the partial isometries $\zeta S_{b^i a}$, $0 \le i < d$.  We have
\[
(\zeta S_{b^i a}) (\zeta S_b) (\zeta S_{b^i a})^* 
= (\zeta S_{b^i a}) (\zeta S_b) (S_{b^i a}^*) 
= S_{b^i a} S_b S_{b^i a}^*.
\]
Thus $\widehat{\gamma}_* ([S_b]) = [\zeta S_b] = [S_{b^i a} S_b S_{b^i a}^*]$; i.e. $\widehat{\gamma}_*(1) = (\pm c)^{-1}$.  Thus $\widehat{\gamma}_*$ is given on $K_1$ by multiplication by $(\pm c)^{-1}$.  The Pimsner-Voiculescu exact sequence for $A \sim (A \times_\gamma \IT) \times_{\widehat{\gamma}} \IZ$ gives

\[
\begin{matrix}
\IZ[\tfrac{1}{c}] & \xrightarrow{\pm c - 1} & \IZ[\tfrac{1}{c}] & \longrightarrow & K_1(A) \\
\uparrow & & & & \downarrow \\
K_0(A) & \longleftarrow & \IZ[\tfrac{1}{d}] & \xleftarrow{d - 1} & \IZ[\tfrac{1}{d}] \\
\end{matrix} 
\]
\vspace*{.05 in}

The various cases of the theorem follow from this diagram.  The identification of the class of the identity in $K_0$ follows from the form of the generators given in Theorem \ref{t.kthyfixedpoint}.
\end{proof}

We end by deriving the essential properties of $C^*(\Lambda)$ from properties of the groupoid $G|_{\partial \Lambda}$. 

\begin{Theorem}
\label{t.groupoidproperties}
\begin{enumerate}
\item $G|_{\partial \Lambda}$ is minimal.
\label{t.groupoidpropertiesa}
\item $G|_{\partial \Lambda}$ is contractive if and only if $d > 1$.
\label{t.groupoidpropertiesb}
\item $G|_{\partial \Lambda}$ is topologically free if and only if $d \nmid c$.
\label{t.groupoidpropertiesc}
\end{enumerate}
\end{Theorem}

\begin{proof}
\eqref{t.groupoidpropertiesa}:  This follows from Theorem 10.14 of \cite{spi}, since $\Lambda$ has only one vertex.

\eqref{t.groupoidpropertiesb}:  We use Theorem 10.16 of \cite{spi}.  Since $\Lambda$ has only one vertex, every nontrivial path is a cycle. Any element which does not by itself form an exhaustive set will be a non-exhaustive cycle in $\Lambda$.  If $d > 1$, then $a$ is such an element.  Conversely, if $d = 1$, then the boundary of $\Lambda$ reduces to a point, and then it is clear that $G|_{\partial \Lambda}$ is not contractive (or even locally contractive).

\eqref{t.groupoidpropertiesc}:  For the \textit{only if} direction, note that if $d | c$, then $b^d a = a b^c = a (b^d)^{(\tfrac{c}{d})}$.  Hence for any $\gamma \in \Lambda$ we have $b^d \gamma = \gamma b^{\theta(\gamma) \tfrac{c}{d}}$.  Thus $b^d \gamma \Cap \gamma$ for all $\gamma$, so that $\Lambda$ has $\{b^d,e\}$-periodicity (as in \cite{spi}, Definition 10.8).  By \cite{spi}, Theorem 10.10, $G|_{\partial \Lambda}$ is not topologically free.  For the converse, suppose that $d \nmid c$.  Let $\alpha \not= \beta$.  We must find $\gamma$ such that $\alpha\gamma \perp \beta\gamma$ (as in \cite{spi}, Remark 10.11).  If $\alpha \perp \beta$, we may take $\gamma = e$.  So suppose that $\alpha \Cap \beta$.  By Proposition \ref{p.qlo} and left-cancellation we may assume that, say, $\theta(\alpha) = 0$.  We treat three cases.  (In the following, when we write $\pm c$ we mean $+c$ in cases (BS1) and (BS2), and $-c$ in case (BS3).)  First, suppose that $\alpha = e$ and $\theta(\beta) > 0$.  Then $\beta = b^{i} a \beta'$, where $i \in [0,d)$.  Let $j \in [0,d)$ with $j \not= i$ (since $d \not= 1$).  Then $\alpha b^j a = b^j a \perp b^i a \beta' b^j a = \beta b^j a$.  Second, suppose that $\alpha = e$ and $\theta(\beta) = 0$.  Then $\beta = b^q$ with $q > 0$.  Since $d \nmid c$, there is a least positive integer $k$ such that $q \bigl( \tfrac {c} {d} \bigr)^k \not\in \IZ$.  Then $q \bigl( \tfrac {\pm c} {d} \bigr)^{k-1} = j + rd$, where $0 < j < d$.  We have
\[
b^q a^k = a^{k-1} b^{q \bigl( \tfrac {\pm c} {d} \bigr)^{k-1}} a = a^{k-1} b^j a b^{rc} \perp a^k.
\]
Finally we suppose that $\alpha = b^p$ and $\theta(\beta) > 0$.  Then $\beta = a \beta'$ (otherwise we could cancel some element of $B$ on the left).  Let $p = i + rd$ with $i \in [0,d)$.  If $i \not= 0$, then $\alpha a = b^i a b^{\pm rc} \perp a \beta' a = \beta a$.  If $i = 0$, then $\alpha ba = b a b^{\pm rc} \perp a \beta' ba = \beta ba$.
\end{proof}

\begin{Corollary}
\label{c.cstarproperties}
$C^*(\Lambda)$ is a Kirchberg algebra if and only if $d \nmid c$.
\end{Corollary}

\end{document}